\documentclass[11pt]{amsart}

\usepackage{a4wide}
\usepackage{enumerate}
\usepackage{graphics}
\usepackage{amssymb}

\newcommand{\ds}{\displaystyle}

\newtheorem{Th}{Theorem}[section]
\newtheorem{Lem}[Th]{Lemma}
\newtheorem{Co}[Th]{Corollary}
\newtheorem{Def}[Th]{Definition}
\newtheorem{Prop}[Th]{Proposition}

\newtheorem{remark}[Th]{Remark}

\def\CA{{\mathcal A}}
\def\CB{{\mathcal B}}
\def\CD{{\mathcal D}}
\def\CE{{\mathcal E}}
\def\CI{{\mathcal I}}
\def\CO{{\mathcal O}}
\def\CR{{\mathcal R}}
\def\CS{{\mathcal S}}
\def\CT{{\mathcal T}}

\def\S{{\mathbb S}}

\def\N{{\mathbb N}} 
\def\R{{\mathbb R}}

\DeclareMathOperator{\supp}{supp}

\DeclareMathOperator{\graph}{graph}

\DeclareMathOperator{\diag}{diag}

\def\Op{\mathop{\rm Op}\nolimits}

\def\<{\langle}
\def\>{\rangle}
\def\bra{\langle}
\def\ket{\rangle}

\newcommand{\Subsection}[1]{\subsection{ #1} ${}^{}$}

\makeatletter
\@addtoreset{equation}{section}
\makeatother

\author{Ivana Alexandrova}
\address{Ivana Alexandrova, Department of Mathematics, East 
Carolina University, Greenville, NC 27858, USA}
\email{alexandrovai@ecu.edu}
\author{Jean-Fran\c{c}ois Bony}
\address{Jean-Fran\c{c}ois Bony, Institut de Math\'ematiques de Bordeaux, (UMR CNRS 5251), Universit\'e de Bordeaux I, 33405 Talence, France}
\email{bony@math.u-bordeaux1.fr}
\author{Thierry Ramond}
\address{Thierry Ramond, Math\'ematiques, Universit\'e Paris Sud, (UMR CNRS 8628), 91405 Orsay, France}
\email{thierry.ramond@math.u-psud.fr}

\title{Resolvent and Scattering Matrix at the Maximum of the Potential}



\begin{document}
\begin{abstract}
We study the microlocal structure of the resolvent of the 
semi-classical Schr\"{o}din\-ger 
operator with short range 
potential at an energy which is a unique non-degenerate global maximum of the potential.
We prove that it is a semi-classical Fourier integral operator 
quantizing the incoming 
and outgoing Lagrangian submanifolds associated to the fixed hyperbolic point.
We then discuss two applications of this result to describing the
structure of the spectral function and the scattering matrix of the
Schr\"{o}dinger operator at the critical energy.
\end{abstract}
\maketitle

%\tableofcontents

\section{Introduction} \label{Sintro}

We consider the semiclassical Schr\"odinger operator 
\begin{equation} \label{a43}
P = P_{0} +V,\quad  P_{0} =-\frac{1}{2}h^{2}\Delta,\quad 0<h\ll 1,
\end{equation}
where $V\in C^{\infty}(\mathbb{R}^{n}; \mathbb{R})$, $n>1,$ is a short range 
potential, {\it i.e.}, for some $\rho>1$ and
all $\alpha\in\mathbb{N}^{n}$  
\begin{equation}\label{potential}
\left|\partial^{\alpha}V(x)\right|\leq
C_{\alpha}\bra x\ket^{-\rho- \vert \alpha \vert}, \quad
x\in\mathbb{R}^{n}.
\end{equation}
Then $P$ and $P_0$ admit unique self-adjoint realizations on $L^{2}(\mathbb{R}^{n})$ with domain $H^{2}(\mathbb{R}^{n})$, that we still denote $P$ and $P_0$. 
In this paper, we are interested in the microlocal structure of the resolvent and of the spectral measure of $P$, as well as that of the scattering matrix, at energies which are within $\CO(h)$ of a unique non-degenerate global maximum of the potential. More precisely, we show below that they are  semiclassical Fourier integral operators (for short $h$-FIOs). We confer to Appendix \ref{scanal} and to the references given therein for a short presentation of the theory of  such operators.

The resolvent $\CR(E\pm i0 )$ can be defined thanks to  the limiting absorption principle which states that, for 
$E>0$ and when $\alpha>\frac12$, the limit
\begin{equation*}
\CR(E\pm i0)= \lim_{\varepsilon\searrow 0} (P -\left(E\pm i\varepsilon\right))^{-1}
\end{equation*}
exists in $\mathcal{B} ( L_{\alpha}^{2}(\mathbb{R}^{n}), L_{-\alpha}^{2}(\mathbb{R}^{n}) )$, where
$L_{\alpha}^{2} (\mathbb{R}^{n})=\{f; \  \langle x \rangle^{\alpha} f (x) \in
L^{2}(\mathbb{R}^{n})\}$.
We denote $d\CE_{E}$  the spectral measure of $P$. The spectral function $e_{E}$ is the Schwartz kernel of $\ds\frac{d \CE_{E}}{dE}$, and can be represented through the well-known Stone formula
\begin{equation}
\label{stone}
\frac{d \CE_{E}}{dE}=\frac{1}{2i \pi} \left(\CR(E +i0)-\CR(E-i0)\right),\quad E>0.
\end{equation}

The scattering matrix  ${\mathcal S}(E)$ is defined by means of the wave operators.
We recall that under the assumption \eqref{potential}, the wave operators, defined through the strong limits in $L^2$,
\begin{equation}
\label{waveops}
W_{\pm}= \mathop{\text{s--lim}}_{t\to\pm\infty}e^{- i t P /h} e^{i tP_{0}/h}
\end{equation}
exist and are complete. The scattering operator is then defined as $S=W_{+}^{*}W_{-}: L^{2} (\R^{n}) \to L^{2} (\R^{n})$, and $\CS(E,h) : L^{2} (\S^{n-1}) \to L^{2} (\S^{n-1})$ is given by
\begin{equation*}
S=
\int_{\R^+}^{\oplus}F_{0}(E, h)^{-1}\CS(E, h) F_{0}(E, h) \, d E
\end{equation*}
Here $F_{0}(E, h)$ denotes the bounded operator 
from $L^{2}_{\alpha}(\mathbb{R}^{n})$, $\alpha>1/2$,  to $L^{2}(\mathbb{S}^{n-1})$ given by
\begin{equation}
\label{pmlo1}
\left(F_{0}(E, h)f\right)(\omega)= (2\pi h)^{-{n}/{2}} (2E)^{\frac{n-2}{4}} \int_{\mathbb{R}^{n}} e^{-{i}\sqrt{2E}\langle \omega, x\rangle/h}f(x) \, d x , \quad E>0.
\end{equation}

Notice that most of the results on the scattering matrix are given for the operator
\begin{equation}
\CT(E,h)=\frac{1}{2i\pi}({\rm Id} -\CS(E, h)),
\end{equation}
or for the scattering amplitude
\begin{equation}\label{a53}
\CA(E,h)= c_{0} \, K_{\CT(E,h)},
\end{equation}
where we denote $K_{\CT(E,h)}$ the Schwartz kernel of the operator $\CT(E,h)$ and
\begin{equation*}
c_{0} = c_{0} (n, E, h) = - 2 \pi (2E)^{-{(n-1)}/{4}}(2\pi h)^{{(n-1)}/{2}}e^{-i{(n-3)\pi}/{4}}.
\end{equation*}

The semiclassical behavior  of the spectral function for Schr\"{o}din\-ger-like operators has been studied 
extensively.
Popov and Shubin \cite{PoSh83_01}, Popov \cite{Po85_01}, and Vainberg \cite{Va84_01} have
established high energy asymptotics for the spectral function of second
order elliptic operators under the  assumption that these energies  are non-trapping:

\begin{Def}\sl
The energy $E>0$ is non-trapping if for every $(x, \xi)\in p^{-1}(E)\subset
T^{*}\mathbb{R}^{n}$  we have
\begin{equation*}
\lim_{t\to\pm\infty} \vert \exp(tH_p)(x,\xi)\vert =\infty.
\end{equation*}
Here $p(x, \xi)=\frac12\xi^2+V(x)$ denotes  the principal symbol of $P$, and
\begin{equation*}
H_p=\sum_{j=1}^{n}\left(\frac{\partial p}{\partial 
\xi_{j}}\frac{\partial}{\partial 
x_{j}}-\frac{\partial p}{\partial x_{j}}\frac{\partial}{\partial 
\xi_{j}}\right)
\end{equation*}
is its associated Hamiltonian vector field.
\end{Def}

Robert and Tamura \cite{RoTa88_01} consider the spectral function for
semi-classical Schr\"{o}dinger operators with short range 
potentials and
establish asymptotic expansions at fixed non-trapping energy, and at 
non-critical trapping energies in the sense of a distribution.

The microlocal structure of the spectral function has also been analyzed.
In \cite[Theorem XII.5]{Va89_01} Vainberg establishes a high energy asymptotic expansion of 
the spectral function for compactly supported smooth perturbations of the 
Laplacian assuming that the 
energy 1 is non-trapping. 
This asymptotic expansion is expressed in the form of a Maslov
canonical
operator.

C. G\'erard and Martinez \cite{GeMa89_01} have proved that the spectral function for   
certain
long-range Schr\"{o}dinger operators at non-trapping energies $E$ is
a $h$-FIO associated to the canonical relation
$\left(\cup_{t\in\mathbb{R}}\graph\exp(tH_p)|_{p^{-1}(E)}\right)$.
Near the diagonal $\{ (x, \xi , x, \xi ); \  p(x, \xi)=E\}$ they also 
give the following oscillatory integral representation of
the spectral function
\begin{equation*}
e_{E}(x, y, E, h) = \frac{1}{(2\pi
h)^n}\int_{\mathbb{S}^{n-1}} e^{{i}\varphi(x, y, \omega,
E)/h} a(x, y, \omega, E) d\omega,
\end{equation*}
where $\varphi\in C^{\infty}(\mathbb{R}^{2n}\times\mathbb{S}^{n-1})$ is 
such that 
\begin{equation*}
\big( \frac{\partial \varphi}{\partial x} \big)^2+V(x)=E,
\quad 
\frac{\partial \varphi}{\partial x} \big\vert_{\langle x-y, \omega\rangle=0} =\sqrt{E-V(x)}\omega,
\quad 
\varphi_{\vert_{x=y}} =0.
\end{equation*}

In \cite{Al06_02} the first author has studied the microlocal structure of the 
spectral function restricted away from the diagonal in
$\mathbb{R}^{n}\times\mathbb{R}^{n}$ at trapping energies under the 
assumption of the absence of resonances near the real axis, as well as at 
non-trapping energies.
In these cases the spectral function is shown to be an $h$-FIO
associated to 
$\left(\cup_{t\in\mathbb{R}}\graph\exp(tH_p)|_{p^{-1}(E)}\right)$
near a non-trapped trajectory.
Under a certain geometric assumption \cite{Al06_02} also gives an oscillatory
integral representation of the spectral function of the form
\begin{equation*}
e_{E}(x, y, E) = \int e^{i S(x, y, t) /h} a(x, y, t) \, d t ,
\end{equation*}
where 
\begin{equation*}
S(x, y, t)=\int_{l(t, x, y)} \Big( \frac12 \vert \xi(s) \vert^2 + E -V(x(s)) \Big) \, d s ,
\end{equation*}
is the action over the segment $l(t, x, y)$ of the trajectory which connects $x$ with $y$ at time $t$ and $a\in S_{2n+1}^{\frac{n+3}{2}}(1)$.

The structure of the resolvent in various settings 
has been studied in \cite{Al05_01}, \cite{Al06_01}, 
and \cite{HaWu07_01}.
For compactly supported and short range potentials, the resolvent has been 
shown to be a $h$-FIO associated to the Hamiltonian flow relation of the 
principal symbol of $P$ 
restricted to the energy surface in \cite{Al05_01} and \cite{Al06_01}.
Hassell and Wunsch have studied in \cite{HaWu07_01} the resolvent on asymptotically conic non-trapped manifolds. This class contains in particular 
some asymptotically Euclidean spaces after compactification. They prove that the Schwartz kernel of the resolvent is a Legendrian distribution, that is, roughly speaking, a semiclassical Lagrangian distribution where the semiclassical parameter is the distance to the boundary.

The semiclassical behavior  of the scattering amplitude has also been of significant interest to researchers in 
mathematical
physics. It is well known that  ${\mathcal{A}(E, 
h)}$ satisfies ${\mathcal{A}(E, h)}\in
C^{\infty}(\mathbb{S}^{n-1}\times\mathbb{S}^{n-1}
\backslash\diag(\mathbb{S}^{n-1}\times\mathbb{S}^{n-1}))$. Several authors have proved asymptotic expansions for ${\mathcal{A}(E, h)}$, showing in particular a direct relation with the underlying classical mechanics.

To describe these results, let us recall that, for $(a, b) \in T^{*} \R^{n} \setminus \{ 0 \} = \R^{n} \times \R^{n} \setminus \{ 0 \}$, there is a unique bicharacteristic curve ({\it i.e.} an integral curve of $H_p$)
\begin{equation}
\gamma_{\pm}(t, a, b)=(x_{\pm}(t, a,b) , \xi_\pm(t, a, b)) ,
\end{equation}
such that
\begin{equation}
\begin{aligned}
&\lim_{t\to \pm\infty} \vert x_{\pm} (t,a,b) - b t- a \vert = 0   \\
&\lim_{t\to \pm\infty} \vert \xi_{\pm} (t,a,b)- b \vert = 0.
\end{aligned}
\end{equation}
Moreover
\begin{equation} \label{a47}
\left\{ \begin{aligned}
&T^{*} \R^{n} \setminus \{ 0 \} \longrightarrow && T^{*} \R^{n}  \\
&(a, b) && \gamma (0, a, b)
\end{aligned} \right.
\end{equation}
is a $C^{\infty}$ symplectic diffeomorphism onto its image (see \cite[Section XI.2]{ReSi79_01}).

On the other hand, if a bicharacteristic curve $(x (t ,\rho ) , \xi ( t , \rho )) = \exp (t H_{p} ) ( \rho )$ of positive energy satisfies $\vert x(t , \rho) \vert \to + \infty$ as $t \to+ \infty$, there is $(x_{\infty} , \xi_{\infty} ) = (x_{\infty} ( \rho ) , \xi_{\infty}  (\rho ) )\in T^{*} \R^{n}$ such that
\begin{equation}\label{a50}
\begin{aligned}
&\lim_{t\to +\infty} \vert x (t, \rho) -\xi_{\infty}  t- x_{\infty}  \vert = 0,   \\
&\lim_{t\to +\infty} \vert \xi (t, \rho)- \xi_{\infty}  \vert = 0.
\end{aligned}
\end{equation}
In that case
\begin{equation}\label{a49}
\begin{aligned}
\Theta ( \rho ) =& \frac{\xi_{\infty} }{\vert \xi_{\infty}  \vert}  \in \S^{n-1}  \\
Z ( \rho ) =& x_{\infty} - \< x_{\infty}, \xi_{\infty}  \> \frac{\xi_{\infty} }{\vert \xi_{\infty}  \vert^{2}} \in \Theta^{\perp} \sim \R^{n-1} ,
\end{aligned}
\end{equation}
are called the outgoing {(asymptotic) direction} and  outgoing {impact factor}, respectively.

In particular, for given $E>0$, $\alpha\in \S ^{n-1}$ and $z\in\alpha^{\perp}$ (the impact plane), we define
\begin{equation} \label{a19}
\gamma_{\pm}(t,\alpha,z,E)=(x_{\pm}(t,\alpha,z,E),\xi_\pm(t,\alpha,z,E)) := \gamma_{\pm}(t, z, \sqrt{2E} \alpha ).
\end{equation}
If for some $( \omega , z_{-} ) \in T^* \S^{n-1}$, we have $\vert x_{-} (t,\omega,z_-,E) \vert \to \infty$ as $t\to +\infty$, we denote $x_{\infty}(\omega , z_{-} , E),\xi_{\infty} (\omega , z_{-} , E)$ the quantities defined through \eqref{a50} for the curve $\gamma_{-}(t,\omega,z_{-},E)$. We also denote
\begin{equation}\label{ztheta}
\left\{ \begin{aligned}
&\theta = \theta ( \omega ,z_{-} ,E)= \Theta ( \gamma_{-} (0, \omega , z_{-} , E ))   \\
& z_+=z_+(\omega,z_-,E)=Z( \gamma_{-} (0, \omega , z_{-} ,E )) ,
\end{aligned} \right.
\end{equation}
and we shall say that the trajectory $\gamma_-(t,\omega,z_-,E)$ has initial direction $\omega$ and final direction $\theta$,  or that it is  an $(\omega,\theta)$-trajectory.

\begin{Def}\sl
The outgoing direction $\theta\in \S ^{n-1}$ is called 
 regular  for
the incoming direction $\omega\in\S ^{n-1}$, or $\omega$-regular,  if $\theta\ne\omega$ and, for all $z'\in\omega^{\perp}$ with
$\xi_{\infty}(\omega,z', E)=\sqrt{2E}\theta$, the map $\omega^{\perp}\ni z\mapsto \xi_{\infty}(\omega,z, E)\in\S ^{n-1}$ is non-degenerate at $z'$, i.e. $\widehat{\sigma} (z')\neq 0$ where
$$
\widehat{\sigma} (z')=\vert\det (\xi_\infty(\omega,z',E), \partial_{z_1}\xi_\infty(\omega,z',E),\dots,\partial_{z_{n-1}}\xi_\infty(\omega,z',E))\vert.
$$
\end{Def}

Under the assumption that a certain final direction $\theta$ is
regular for a given initial direction $\omega$, it has been shown that
\begin{equation}\label{vexpansion}
{\mathcal{A}(E, h)}(\theta,
\omega)=\sum_{j=1}^{l}\hat{\sigma}(\omega,z_{j},
E)^{-1/2}\exp (ih^{-1}S_{j}-i\mu_{j}\pi/2 )+\mathcal{O} (h),
\end{equation}
where
$$
\left(z_{j}\right)_{j=1}^{l}=\big(\xi_{\infty}^{-1}(\sqrt{2E}\omega,
\cdot ,E)\big) (\theta),
$$
and
\begin{equation}\label{modaction}
S_j=\int_{-\infty}^{\infty}
\big(
\vert \xi_{-}(t,\omega,z_{j},E)
\vert^{2} -2E\big) d t
-
\langle
x_{\infty}(\omega, z_{j},E),
\sqrt{2E}\theta(\omega,z_j,E)
\rangle
\end{equation} 
is a modified action along the $j$-th $(\omega, \theta)$-trajectory,
and $\mu_{j}$ is the Maslov index of that trajectory.
Such a result has been obtained by Vainberg \cite{Va84_01}, who has studied smooth compactly
supported potentials $V$ at energies
$E>\sup V$.
Guillemin \cite{Gu76_01} has established a similar asymptotic expansion in the setting of smooth
compactly-supported metric perturbations of the Laplacian.
Working with trapping potential perturbations of the Laplacian satisfying
\eqref{potential} with $\rho>\max\left(1, \frac{n-1}{2}\right)$, Yajima
\cite{Ya87_01} has proved such an asymptotic expansion in the $L^{2}$ sense.
For non-trapping short-range ($\rho>1$) potential perturbations of the Laplacian, Robert and
Tamura \cite{RoTa89_01} have proven that \eqref{vexpansion} holds 
pointwise.
Their result has been extended to the case of
trapping energies by
Michel \cite{Mi04_01} under an additional assumption on the distribution of the resonances of $P$.

First to study the microlocal structure of the scattering amplitude was 
Protas \cite{Pr82_01}.
He has shown that at non-trapping energies and for fixed initial
directions the scattering amplitude is a Maslov canonical 
operator associated to  some natural Lagrangian
submanifolds of $T^{*}\mathbb{S}^{n-1}$.
This representation of the scattering amplitude is shown to hold
uniformly in an open set
containing the final direction and disjoint from the initial direction.

In \cite{Al05_01} and  \cite{Al06_01} the first author has proved, without making the non-degeneracy assumption, that for short-range Schr\"odinger operators satisfying a polynomial estimate for their resolvent, the scattering
amplitude  is an $h$-FIO
associated to the scattering relation microlocally near a non-trapped trajectory.
The scattering relation for a short range potential at an energy $E>0$ is defined near a 
non-trapped trajectory as follows.
If $\gamma_0: t\mapsto\gamma_-(t,\omega_0,z_0,E)$ is non-trapped, there exists an open set $U\subset T^{*}\mathbb{S}^n$ with $(\omega_0,
z_0)\in U$ such that for every $(\omega, z_-)\in U$ the trajectory
$t\mapsto\gamma_-(t,z, \omega_0,E)$ is 
non-trapped. The scattering relation near $\gamma_0$ is given by (see Figure \ref{fig:sr})
\begin{equation}\label{a52}
{\mathcal {SR}}(E)= \{ (\theta(\omega,z_-,E), -\sqrt{2E}z_+(\omega,z_-,E),\omega, -\sqrt{2E}z_-); \  (\omega,z_-)\in U \} ,
\end{equation}
where $\theta$ and $z_+$ are defined in (\ref{ztheta}).

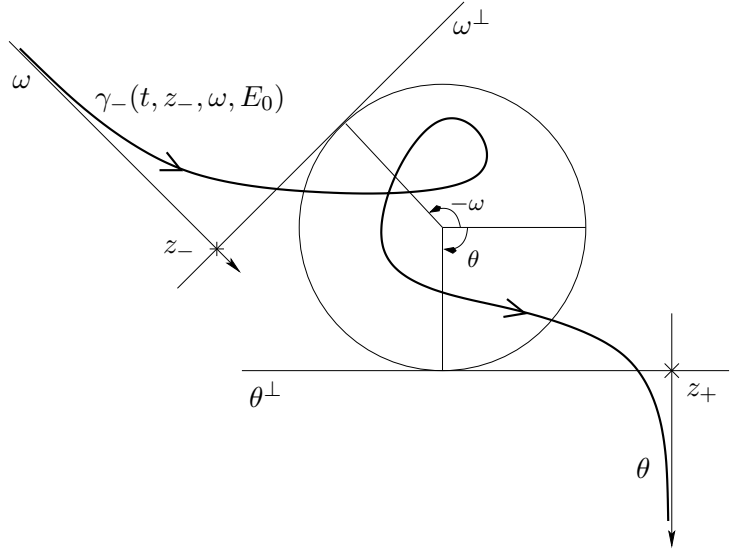
\begin{figure}\label{fig:sr}
\begin{center}
\begin{picture}(0,0)%
\includegraphics{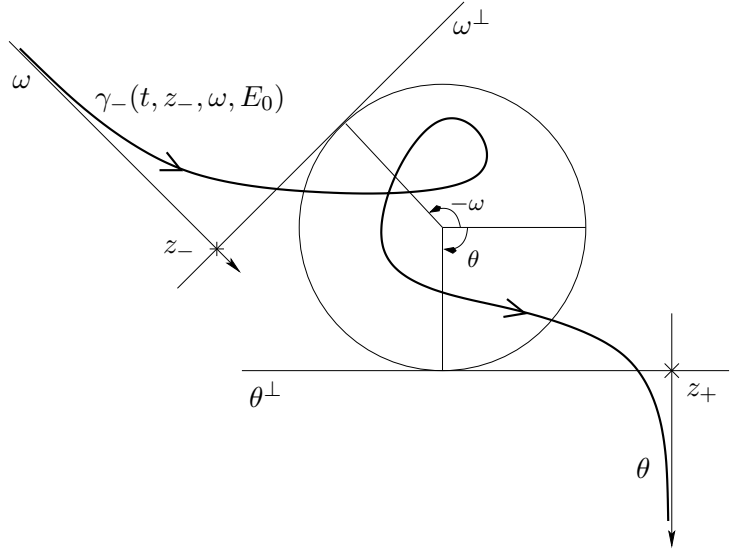}%
\end{picture}%
\setlength{\unitlength}{1184sp}%
\begingroup\makeatletter\ifx\SetFigFont\undefined%
\gdef\SetFigFont#1#2#3#4#5{%
  \reset@font\fontsize{#1}{#2pt}%
  \fontfamily{#3}\fontseries{#4}\fontshape{#5}%
  \selectfont}%
\fi\endgroup%
\begin{picture}(15119,11519)(-4296,-10733)
\put(5326,-4786){\makebox(0,0)[lb]{\smash{{\SetFigFont{9}{10.8}{\rmdefault}{\mddefault}{\updefault}$\theta$}}}}
\put(5026, 89){\makebox(0,0)[lb]{\smash{{\SetFigFont{11}{13.2}{\rmdefault}{\mddefault}{\updefault}$\omega^{\perp}$}}}}
\put(9901,-7486){\makebox(0,0)[lb]{\smash{{\SetFigFont{11}{13.2}{\rmdefault}{\mddefault}{\updefault}$z_{+}$}}}}
\put(-2474,-1411){\makebox(0,0)[lb]{\smash{{\SetFigFont{11}{13.2}{\rmdefault}{\mddefault}{\updefault}$\gamma_{-} ( t, z_{-} , \omega , E_{0} )$}}}}
\put(8851,-9211){\makebox(0,0)[lb]{\smash{{\SetFigFont{11}{13.2}{\rmdefault}{\mddefault}{\updefault}$\theta$}}}}
\put(751,-7711){\makebox(0,0)[lb]{\smash{{\SetFigFont{11}{13.2}{\rmdefault}{\mddefault}{\updefault}$\theta^{\perp}$}}}}
\put(-1049,-4486){\makebox(0,0)[lb]{\smash{{\SetFigFont{11}{13.2}{\rmdefault}{\mddefault}{\updefault}$z_{-}$}}}}
\put(-4199,-1036){\makebox(0,0)[lb]{\smash{{\SetFigFont{11}{13.2}{\rmdefault}{\mddefault}{\updefault}$\omega$}}}}
\put(4951,-3586){\makebox(0,0)[lb]{\smash{{\SetFigFont{9}{10.8}{\rmdefault}{\mddefault}{\updefault}$-\omega$}}}}
\end{picture}%
\end{center}
\caption{The scattering relation consists of the points
$(\theta,-\sqrt{2E_{0}}z_+,\omega, -\sqrt{2E_{0}}z_-)$ related as in this figure.}
\end{figure}

It is also explained in \cite{Al05_01} how the expansion \eqref{vexpansion} follows from this result
once the non-degeneracy assumption on the initial and final directions is made.
The asymptotic expansion obtained is more
general than the one given in \eqref{vexpansion} in that it holds microlocally
near $\left(\omega, \theta\right)$ trajectories and not only for fixed
initial and final directions.

In the context of scattering on manifold with boundary, Hassell and Wunsch \cite{HaWu07_01} have shown that the scattering matrix at non-trapping energies is a Legendrian-Lagrangian
distribution associated to the total sojourn relation.
In \cite{Va07_01}, Vasy has also studied the scattering matrix on asymptotically De Sitter-like spaces (a large class of non-trapped spaces with two asymptotically hyperbolic ends). Under the assumption that the bicharacteristic curves go from one end to the other, he has proved that the scattering matrix is a FIO associated to the natural relation between these two ends.

In this paper we continue the study of the scattering matrix for energies which are within $\CO(h)$ of a unique non-degenerate global maximum of the potential. In that setting, in the one-dimensional case, the scattering matrix is a 2 by 2 matrix, and the semiclassical expansion of its coefficient has been given by the third author in \cite{Ra96_01}. The computations there rely on complex WKB constructions for the generalized eigenfunctions, as well as a microlocal reduction to a normal form near the maximum point.

For such a critical energy, we have already  studied the scattering
amplitude in the  $n$-dimensional case:  In \cite{AlBoRa07_01}, we have established the semiclassical expansion  of the scattering amplitude. In that paper, we use Robert and Tamura's formula (see \eqref{Gu76_01} below) for the scattering amplitude. This formula itself relies on Isosaki and Kitada's construction of a suitable approximation for the wave operators, and, roughly speaking, reduces the problem to that of the description of generalized 
eigenfunctions in a compact set. To do so, we  essentially follow the study in \cite{BoFuRaZe07_01}, to obtain such a description in a neighborhood of the critical point. 

In the present paper  we describe the microlocal structure of the spectral 
function and of the scattering matrix at such energies. More precisely we  show that they are  $h$-FIO's associated to quite canonical relations.  To the contrary of \cite{AlBoRa07_01}, we do not suppose the non-degeneracy assumption, and we state no geometrical assumptions concerning the behavior of the incoming and outgoing stable manifolds at infinty. However the results below are valid in a somewhat smaller region of the phase space. Of course one recovers  parts of the results of \cite{AlBoRa07_01} in that smaller region once the geometric assumptions alluded to above are made.

\section{Assumptions and main results}

We suppose that the potential $V$ is a short-range, $C^\infty$ function on $\R^n$ (see (\ref{potential})),  and we make the following further assumptions:

\begin{itemize}
\item[{\it (A1)}]\label{a1}
$V$ has a non-degenerate global maximum at $x=0$, with $V (0) =E_0 >0$. We can always suppose that
\begin{equation*}
V(x)=E_0-\sum_{j=1}^{n}\frac{\lambda_{j}^{2}}{2} x_{j}^{2}
+\mathcal{O}(x^{3}), \quad x \to 0,
\end{equation*}
where $0 < \lambda_{1} \leq \lambda_{2} \leq\ldots \leq \lambda_{n}$.

\item[{\it (A2)}]\label{a2} The trapped set at energy $E_0$ is reduced to $(0,0)$, namely
\begin{equation*}
\{(x, \xi)\in p^{-1}(E_0) ; \ \exp\left(t H_p\right)\left(x, 
\xi\right)\nrightarrow \infty \text{ as } 
t\to \pm\infty\}=\{(0, 0)\}.
\end{equation*}
\end{itemize}
Then, the linearized vector field of
$H_{p}$ at $\left(0, 0\right)$ is
\begin{equation*}
d_{\left(0, 0\right)}H_{p}=\left(
\begin{array}{cc}
0& {\rm Id} \\
{\diag(\lambda_{1}^{2},
\dots, \lambda_{n}^{2})}&0
\end{array}
\right ).
\label{linearise}
\end{equation*}
and, by the stable/unstable manifold theorem, there exist Lagrangian submanifolds
$\Lambda_{\pm}$ of $T^* \mathbb{R}^{n}$ (see Figure \ref{fig:lambdapm}) satisfying
\begin{equation*}
\Lambda_{\pm}=\left\{(x,\xi)\in T^* \mathbb{R}^{n} ; \ \exp ( tH_{p}) (x,\xi)\to
\left(0, 0\right)
\text{ as } t\to \mp\infty\right\}\subset p^{-1}(E_0)
\end{equation*}

Notice that the assumptions {\it (A1)} and {\it (A2)} imply that $V$ has an absolute global maximum at $x=0$. Indeed, if ${\mathcal L} = \{ x\neq 0; \ V (x) \geq E_{0} \}$ was non empty, the geodesic, for the Agmon distance $(E_{0} - V(x))_{+}^{1/2} dx$, between $0$ and ${\mathcal L}$ would be  the projection of a trapped bicharacteristic (see \cite[Theorem 3.7.7]{AbMa78_01}).

\begin{figure}[!h]\label{fig:lambdapm}
\begin{center}
\begin{picture}(0,0)%
\includegraphics{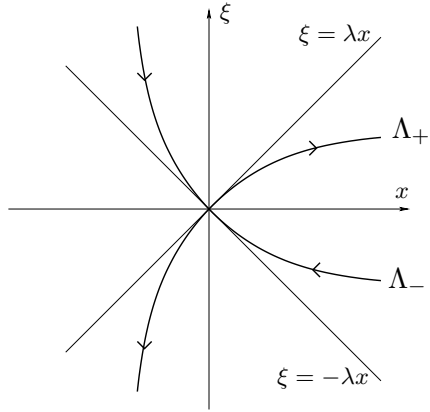}%
\end{picture}%
\setlength{\unitlength}{1184sp}%
\begingroup\makeatletter\ifx\SetFigFont\undefined%
\gdef\SetFigFont#1#2#3#4#5{%
  \reset@font\fontsize{#1}{#2pt}%
  \fontfamily{#3}\fontseries{#4}\fontshape{#5}%
  \selectfont}%
\fi\endgroup%
\begin{picture}(8444,8444)(1779,-8183)
\put(7426,-7636){\makebox(0,0)[lb]{\smash{{\SetFigFont{9}{10.8}{\rmdefault}{\mddefault}{\updefault}$\xi = - \lambda x$}}}}
\put(9826,-2461){\makebox(0,0)[lb]{\smash{{\SetFigFont{11}{13.2}{\rmdefault}{\mddefault}{\updefault}$\Lambda_{+}$}}}}
\put(9751,-5611){\makebox(0,0)[lb]{\smash{{\SetFigFont{11}{13.2}{\rmdefault}{\mddefault}{\updefault}$\Lambda_{-}$}}}}
\put(9901,-3736){\makebox(0,0)[lb]{\smash{{\SetFigFont{9}{10.8}{\rmdefault}{\mddefault}{\updefault}$x$}}}}
\put(6226, 14){\makebox(0,0)[lb]{\smash{{\SetFigFont{9}{10.8}{\rmdefault}{\mddefault}{\updefault}$\xi$}}}}
\put(7876,-436){\makebox(0,0)[lb]{\smash{{\SetFigFont{9}{10.8}{\rmdefault}{\mddefault}{\updefault}$\xi = \lambda x$}}}}
\end{picture}%
\end{center}
\caption{The incoming $\Lambda_{-}$ and outgoing $\Lambda_{+}$ 
Lagrangian 
submanifolds}
\end{figure}
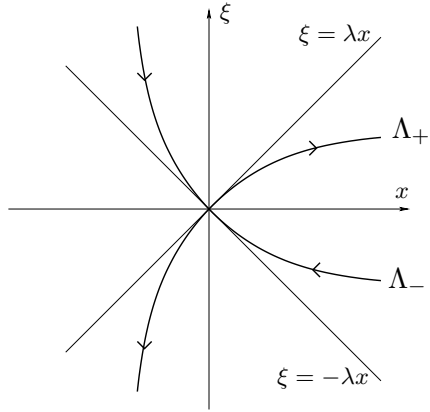

We recall from \cite{HeSj85_01} that if
$\rho_\pm\in\Lambda_\pm$ and $(x_\pm(t, \rho_\pm), \xi_\pm(t, \rho_\pm) ) = \exp (tH_p ) (\rho_\pm)$ is the bicharacteristic starting from $\rho_\pm$,
then for some
$g_\pm\in C^{\infty}\left(\Lambda_{\pm}; \mathbb{R}^{n}\right)$ and 
$\varepsilon>0$, 
\[x_\pm(t; 
\rho_\pm)=g_\pm(\rho_\pm)e^{\pm\lambda_{1}t}+\mathcal{O}(e^{\pm
(\lambda_{1}+\varepsilon)t}) \text{ as } t\to\mp\infty.\]
We let 
\begin{equation*}
\widetilde{\Lambda_{+}\times\Lambda_{-}}=
\big\{(\rho_{+}, \rho_{-})\in\Lambda_{+}\times\Lambda_{-} ; \
\langle 
g_{+}(\rho_{+}), g_{-}(\rho_{-})\rangle\ne 
0\big\},
\end{equation*}
and define $\widetilde{\Lambda_{-}\times\Lambda_{+}}$ 
analogously.
\begin{remark}\sl
The reader may notice that if $\lambda_2 > \lambda_1$, then, by \cite[(6.96)]{AlBoRa07_01}, 
the vectors $g_{\pm}(\rho)$ are for any $\rho$ collinear with $(1, 0, \dots, 0)\in\mathbb{R}^{n}.$   
Therefore, $\widetilde{\Lambda_+\times\Lambda_-}=\Lambda_+ \setminus 
\widetilde{\Lambda_+} \times\Lambda_-
\setminus \widetilde{\Lambda_-}$, where $\widetilde{\Lambda_\pm} = \{\rho\in\Lambda_\pm ; \  g_\pm(\rho)=0\}.$ 
We recall from \cite{BoFuRaZe07_01} that in this case $\dim 
\widetilde{\Lambda_{\pm}}=n-1$.
\end{remark}

Our main result is the following

\begin{Th}\label{ThRes}\sl 
Assume {\it (A1)} and {\it (A2)}. Let $E=E_0 +hE_1$, with  $E_1\in ]-C_0, C_0[$ for some $C_0>0$. 
Then, microlocally near any $(\rho_+,\rho_-)\in \widetilde{\Lambda_+ \times \Lambda_-}$  we have
$$
\CR(E+i0)\in \CI_h^{1-\frac{\sum_{j=1}^{n}\lambda_{j}}{2\lambda_1}} \big( \R^n\times \R^n,\Lambda_+\times \Lambda_- {}' \big),
$$
and, microlocally near any $(\rho_-,\rho_+)\in \widetilde{\Lambda_- \times \Lambda_+}$,
$$
\CR(E-i0)\in \CI_h^{1-\frac{\sum_{j=1}^{n}\lambda_{j}}{2\lambda_1}} \big( \R^n\times \R^n,\Lambda_-\times \Lambda_+ {}' \big) .
$$
\end{Th}

Concerning the spectral function, using Stone's Formula (\ref{stone}), we obtain immediately  the

\begin{Co}\sl
Assume {\it (A1)} and {\it (A2)}. Let $E=E_0+hE_1$ where $E_1\in ]-C_0, 
C_0[$ for some $C_0>0$. Then the spectral function at the energy $E$ 
satisfies, microlocally near $(\rho_1,\rho_2)\in\widetilde{\Lambda_+ \times 
\Lambda_-}\cup\widetilde{\Lambda_-\times\Lambda_+}$,
\begin{equation*}
e_E\in \CI_{h}^{1 -\frac{\sum_{j=1}^{n}\lambda_{j}}{2\lambda_1}}
\big (\R^n\times \R^n, \Lambda_+ \times 
\Lambda_-\cup \Lambda_-\times\Lambda_+ {}' \big).
\end{equation*}
\end{Co}

Now we pass to our result concerning  the scattering matrix.
We denote (see Figure \ref{newscatrel})
\begin{align*}
&\Lambda_{+}^\infty=\{(\theta, - \sqrt{2E_{0}}z_+)\in T^*\S^{n-1}; \ \gamma_+(t,z_+,\theta,E_{0})\in \Lambda_+\},
\\
&\Lambda_{-}^\infty=\{(\omega, - \sqrt{2E_{0}}z_-)\in T^*\S^{n-1}; \ \gamma_-(t,\omega,z_-,E_{0})\in \Lambda_-\}.
\end{align*}
Notice that $\Lambda_{\pm}^\infty$ are submanifolds of $T^*\S^{n-1}$ of dimension $n-1$, since the map $( \alpha ,z) \mapsto \gamma_{\pm} (0,\alpha ,z,E)$ is a $C^{\infty}$ diffeomorphism.
We set also
\begin{align*}
\widetilde{\Lambda_{+}^\infty\times \Lambda_{-}^\infty} = \big\{ (\theta , - \sqrt{2E_{0}} z_+,\omega, - \sqrt{2E_{0}}z_-) & \in \Lambda_{+}^\infty \times \Lambda_{-}^\infty ; \\
\big\< & g_+(\gamma_+(0,z_+,\theta,E_{0} )),g_-(\gamma_-(0,\omega,z_-,E_{0} )) \big\> \neq 0 \big\}.
\end{align*}

\begin{Th}\label{Thsa}\sl 
Assume {\it (A1)} and {\it (A2)}. Let $E=E_0 +hE_1$, with $E_1\in ]-C_0, C_0[$
for some $C_0>0.$
Then, if $\omega\neq \theta$, microlocally near $(\theta,\sqrt{2E_{0}}z_+,\omega,\sqrt{2E_{0}}z_-)\in \widetilde{\Lambda_{+}^\infty\times \Lambda_{-}^\infty}$,
\begin{equation*}
\CS(E,h)\in\mathcal{I}_{h}^{\frac{1}{2} - \frac{\sum_{j=1}^{n}
\lambda_{j}}{2\lambda_{1}}} \big( \mathbb{S}^{n-1}\times\mathbb{S}^{n-1}, 
\Lambda^\infty_{+}\times\Lambda^\infty_{-} {}'\big).
\end{equation*}
\end{Th}

For potentials $V$ with compact support, this result can be extended to the case $\omega = \theta$. In fact, for such potentials there exists a nice representation of the scattering matrix which is valid even for $\omega = \theta$ (see \cite[Equation (46)]{Al05_01}).

Notice that, near non-trapped trajectories, our proof here gives the following improvement of \cite[Main Theorem]{Al06_01} for what concerns the order. The order is here optimal as shown by the results of the paper \cite{RoTa89_01}. Of course, one can obtain analogous results concerning the resolvent or the spectral function (see \eqref{a51}).

\begin{Th}\sl \label{nontrappedcase}
Suppose \eqref{a43}, \eqref{potential}, $E_{0}>0$ and, for some $\alpha>1/2$ and some $N\in\R$,
\begin{equation*}
\Vert \CR ( E_{0}+i0 ) \Vert_{\CB(L^2_{\alpha}(\R^n),L^2_{-\alpha}(\R^n))}=\CO(h^N).
\end{equation*}
If $(\omega,z_{-})\in T^*\S^{n-1}$ is such that $\gamma_{-}(t,\omega,z_{-},E_{0})$ is non-trapped, then, microlocally near $(\theta(\omega,z_{-},E_{0}), \sqrt{2E_{0}}z_{+}(\omega,z_{-},E_{0}),\omega,\sqrt{2E_{0}}z_{-})$, provided $\omega\neq \theta(\omega,z_{-},E_{0})$ we have
\begin{equation*}
\CS (E_{0} ,h)\in \CI^{0}_{h} \big( \S^{n-1}\times \S^{n-1}, \CS\CR (E_{0})' \big).
\end{equation*}
\end{Th}

\begin{figure}\label{newscatrel}
\begin{center}
\begin{picture}(0,0)%
\includegraphics{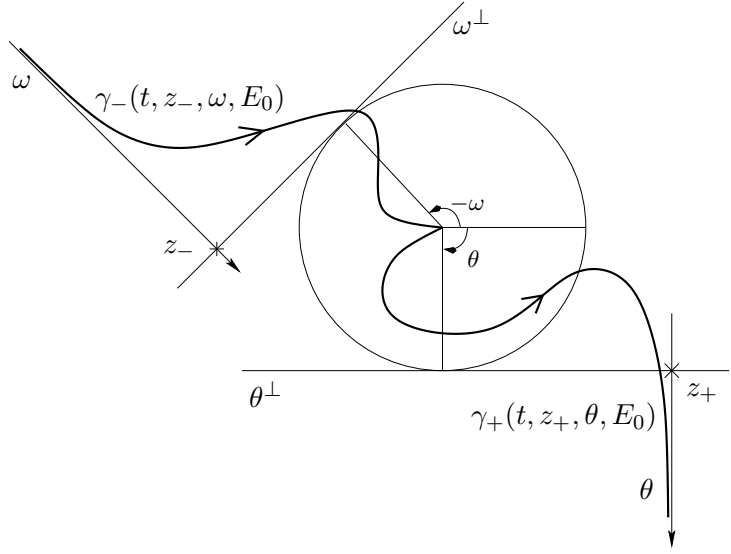}%
\end{picture}%
\setlength{\unitlength}{1184sp}%
\begingroup\makeatletter\ifx\SetFigFont\undefined%
\gdef\SetFigFont#1#2#3#4#5{%
  \reset@font\fontsize{#1}{#2pt}%
  \fontfamily{#3}\fontseries{#4}\fontshape{#5}%
  \selectfont}%
\fi\endgroup%
\begin{picture}(15119,11519)(-4296,-10733)
\put(9901,-7486){\makebox(0,0)[lb]{\smash{{\SetFigFont{11}{13.2}{\rmdefault}{\mddefault}{\updefault}$z_{+}$}}}}
\put(8926,-9661){\makebox(0,0)[lb]{\smash{{\SetFigFont{11}{13.2}{\rmdefault}{\mddefault}{\updefault}$\theta$}}}}
\put(-4199,-1036){\makebox(0,0)[lb]{\smash{{\SetFigFont{11}{13.2}{\rmdefault}{\mddefault}{\updefault}$\omega$}}}}
\put(-2474,-1411){\makebox(0,0)[lb]{\smash{{\SetFigFont{11}{13.2}{\rmdefault}{\mddefault}{\updefault}$\gamma_{-} ( t, z_{-} , \omega , E_{0} )$}}}}
\put(-1049,-4486){\makebox(0,0)[lb]{\smash{{\SetFigFont{11}{13.2}{\rmdefault}{\mddefault}{\updefault}$z_{-}$}}}}
\put(4951,-3586){\makebox(0,0)[lb]{\smash{{\SetFigFont{9}{10.8}{\rmdefault}{\mddefault}{\updefault}$-\omega$}}}}
\put(5326,-4786){\makebox(0,0)[lb]{\smash{{\SetFigFont{9}{10.8}{\rmdefault}{\mddefault}{\updefault}$\theta$}}}}
\put(5026, 89){\makebox(0,0)[lb]{\smash{{\SetFigFont{11}{13.2}{\rmdefault}{\mddefault}{\updefault}$\omega^{\perp}$}}}}
\put(751,-7711){\makebox(0,0)[lb]{\smash{{\SetFigFont{11}{13.2}{\rmdefault}{\mddefault}{\updefault}$\theta^{\perp}$}}}}
\put(5401,-8086){\makebox(0,0)[lb]{\smash{{\SetFigFont{11}{13.2}{\rmdefault}{\mddefault}{\updefault}$\gamma_{+} ( t, z_{+} , \theta , E_{0} )$}}}}
\end{picture}%
\end{center}
\caption{The scattering relation
$\Lambda^{\infty}_{+}\times\Lambda^{\infty}_{-}$ consists of the 
points $(\theta, -\sqrt{2E_{0}}z_+, \omega, -\sqrt{2E_{0}}z_-)$ related as in this figure.}
\end{figure}
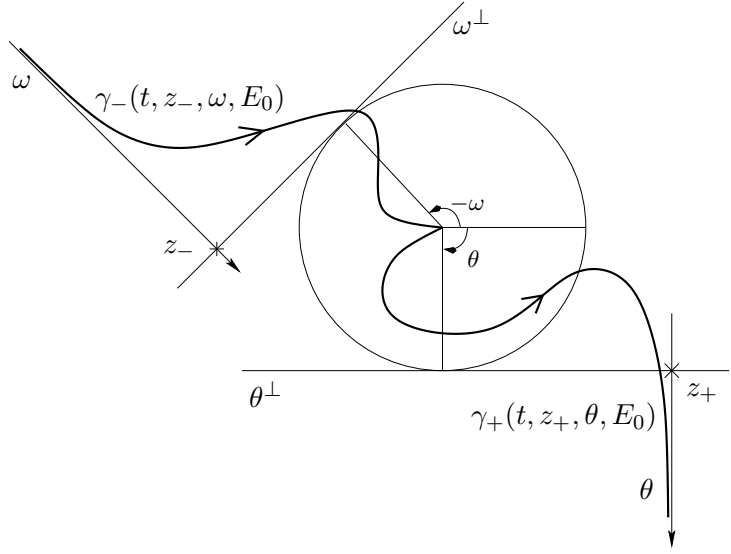

This paper is organized as follows.
We prove Theorem \ref{ThRes} in Section \ref{Sres}, and in Section
\ref{Smspf}, we give the microlocal representations of the resolvent 
and
the spectral function implied by Theorem \ref{ThRes}.
In Section \ref{Sthsa} we prove Theorem \ref{Thsa} using the
representation of the scattering amplitude presented in Section 
\ref{Srepr}. We sketch the proof of Theorem \ref{nontrappedcase} in Section \ref{secntc}.
We use Theorem \ref{Thsa} to deduce an oscillatory integral
representation and an integral representation of the scattering 
amplitude in Section \ref{Smsa}.
Lastly, in Appendix \ref{scanal} we review the notions from
semi-classical analysis most relevant to this work.

\section{The resolvent as a semi-classical Fourier integral operator}

\Subsection{Proof of Theorem \ref{ThRes}}\label{Sres}

We shall prove that $\CR(E+i0)\in\mathcal{I}_{h}^{1-\frac{\sum_{j=1}^{n}\lambda_j}{2\lambda_{1}}} \big( \R^{n} \times \R^{n} , \Lambda_{+}\times\Lambda_{-} {}'\big)$ microlocally near $\left(\rho_{+}, \rho_{-}\right) \in \widetilde{\Lambda_{+} \times \Lambda_{-}}$.
The proof in the case of the incoming resolvent $\CR(E-i0)$ is analogous, and we omit it. The resolvent estimate 
from \cite[Theorem 2.1]{AlBoRa07_01} 
\begin{equation}\label{resest}
\Vert \CR(E\pm i 0)\Vert_{\mathcal{B} (L_{\alpha}^{2} , L_{-\alpha}^{2} )} = \mathcal{O} \left(\frac{\left|\log 
h\right|}{h}\right), \text{ for } \alpha>\frac12,
\end{equation}
and 
\cite[Lemma 1]{Al06_01} give that $K_{\CR(E\pm i0)} \in \CS_{h}' (\mathbb{R}^{2n} )$.
Let $\alpha^{\pm}\in C^{\infty}_{0} (T^*\mathbb{R}^n)$ be supported near $\rho_{\pm}$.
We consider
\begin{equation*}
\mathcal{I}= \Op (\alpha^{+}) \CR (E + i 0) \Op (\alpha^{-}).
\end{equation*}

\begin{Prop}\sl \label{resFIO}
There exist $T_{1}>0$ and $\chi \in C^{\infty}_{0} (]0,+\infty[)$ such that
\begin{equation}\label{a6}
\mathcal{I}= e^{-iT_{1}(P-E)/h} \Op (\alpha^+_{T_{1}}) \mathcal{J}(E) \left(\frac{i}{h} \int \chi(t) e^{-it(P-E)/h} dt\right) e^{iT_{1}(P-E)/h} \Op (\alpha^{-}) + R ,
\end{equation}
where $\Vert R \Vert_{{\mathcal B} ( L^{2} , L^{2} )} = {\mathcal O} (h^{\infty})$, the symbol $\alpha^+_{T_{1}} \in S (\langle x\rangle^{-\infty}\langle \xi\rangle^{-\infty})$ is given by
\begin{equation*}
\Op ( \alpha^+_{T_{1}} ) = e^{iT_{1}(P-E)/h} \Op (\alpha^+) e^{-iT_{1}(P-E)/h} ,
\end{equation*}
and $\mathcal{J}(E)\in\mathcal{I}_{h}^{-\frac{\sum_{j=1}^{n}\lambda_{j}}{2\lambda_1}} \big( \mathbb{R}^{n} \times \R^{n} , \Lambda_+ \times \Lambda_- {}' \big)$ is given by \cite[Theorem 2.6]{BoFuRaZe07_01} and \cite[Remark 2.7]{BoFuRaZe07_01}.
\end{Prop}

It is possible show a better estimate for the remainder term $R$. In fact, we have
\begin{equation*}
\Vert \< x, h D \>^{N} R \< x, h D \>^{N} \Vert_{{\mathcal B} ( L^{2} , L^{2} )} = {\mathcal O} (h^{\infty}) ,
\end{equation*}
for any $N \in \R$.

\begin{proof}   Since $(\rho_{+},\rho_{-})\in \widetilde 
{\Lambda_{+}\times  \Lambda_{-}}$,  one can find $T_{1}>0$ 
such that $\rho_{1}=\exp (-T_{1}H_{p})(\rho_{+})$ belongs to 
$\Lambda_{+}\setminus \widetilde{\Lambda_{+}} (\rho_{-})$ and  
is as close as needed to $(0,0)$.
We have
\begin{align}
\nonumber
\Op (\alpha^{+}) \CR(E+i0) \Op (\alpha^{-})=& e^{-iT_{1}(P-E)/h} 
\Op (\alpha^+_{T_{1}}) e^{iT_{1}(P-E)/h}\CR(E+i0) \Op (\alpha^{-}) \\
=& e^{-iT_{1}(P-E)/h} \Op (\alpha^+_{T_{1}}) \CR(E+i0) e^{iT_{1}(P-E)/h} \Op (\alpha^{-}) .
\label{bzbzb}
\end{align}

We denote $\mathcal{K}=\CR(E+i0)e^{iT_{1}(P-E)/h}\Op(\alpha^{-})$. First 
we observe 
that
\begin{equation*}
(P-E)\mathcal{K} = e^{iT_{1}(P-E)/h} \Op(\alpha^{-}) = 0 \text{ microlocally 
near }(0,0),
\end{equation*}
and we want to apply the results of \cite{BoFuRaZe07_01} in order to compute $\mathcal{K}$ microlocally near $(0,0)$.
Here, and in that follows, we say that an operator $A$ is microlocally $0$ near $V \subset T^{*} \R^{n}$ (respectively $\rho \in T^{*} \R^{n}$) when there exists $\beta \in S (1)$ with $\beta =1$ in a neighborhood of $V$ (respectively $\rho$) such that
\begin{equation*}
\Vert \Op (\beta ) A \Vert_{{\mathcal B} (L^{2} , L^{2})} = {\mathcal O} (h^{\infty}).
\end{equation*}
To that end, we need to know $\mathcal{K}$  microlocally near $\mathcal{S}=\{(x,\xi)\in\Lambda_{-}; \ \vert x\vert=\varepsilon\}$ for some given $\varepsilon>0$ small enough.

We choose $R>0$ such that $e^{iT_{1}(P-E)/h} \Op(\alpha^{-})$ is microlocally 0 out of $B(0,R)$.
One can easily see that there exist $T>0$ and a neighborhood $U$ of $\mathcal{S}$ in $T^*\mathbb{R}^n$, such that
\begin{equation*}
\forall \rho\in U , \ \forall t\geq T, \quad \exp(-tH_{p})(\rho)\notin B(0,R) \times \R^{n} .
\end{equation*}
Now we have
\begin{equation*}
\mathcal{K} = \frac{i}{h} \int_{0}^T e^{-it(P-E)/h} e^{iT_{1}(P-E)/h} \Op (\alpha^{-}) \, d t + e^{-iT(P-E)/h} \mathcal{K},
\end{equation*}
and we claim that the second term of the right hand side vanishes microlocally in $U$.
Indeed, as in \cite[Section 5]{AlBoRa07_01}, one can show that $e^{iT(P-E)/h}\mathcal{K}$ is microlocally $0$ in some incoming region $\Gamma_{-}(R_{0},\sigma,d)$, where we use the standard notation
\begin{equation}\label{gammainout}
\Gamma_{\pm} (R,d, \sigma ) = \big\{ (x, \xi ) \in \R^{n} \times \R^{n} ; \ \vert x \vert > R , \ d^{-1} < \vert \xi \vert < d, \ \pm \cos (x, \xi ) > \pm \sigma \big\} ,
\end{equation}
for  incoming and outgoing regions. Moreover we have
\begin{equation*}
(P-E)e^{- i T(P-E)/h}\mathcal{K}=e^{-i T (P-E)/h}e^{iT_{1}(P-E)/h}\Op(\alpha^{-}) =0 ,
\end{equation*}
microlocally in $\cup_{t\geq 0}\exp( -tH_{p}) U$, and the claim follows by a 
usual propagation of singularity argument.

Thus we have, with the notation of \cite[Section 2]{BoFuRaZe07_01}, microlocally 
near $\rho_{1}$,
\begin{equation*}
\CR(E+i0)e^{iT_{1}(P-E)/h} \Op(\alpha^{-})=\mathcal{J}(E) \left(   
\frac{i}{h}\int_{0}^T e^{-it(P-E)/h} \  dt
\right)e^{iT_{1}(P-E)/h} \Op(\alpha_{-}).
\end{equation*}
Finally, we notice that there exists $\delta>0$ such that, for any $\chi \in C^{\infty}_{0}(]0,T[)$ with $\chi=1$ on $[\delta,T-\delta]$, we have, microlocally near $\rho_{1}$,
\begin{equation*}
\CR(E+i0)e^{iT_{1}(P-E)/h} \Op(\alpha^{-})=\mathcal{J}(E) \left(   
\frac{i}{h}\int \chi(t) e^{-it(P-E)/h}   \  dt \right)
e^{iT_{1}(P-E)/h} \Op(\alpha_{-}).
\end{equation*}
Indeed, by Egorov's theorem,  $e^{-it(P-E)/h} e^{iT_{1}(P-E)/h} \Op(\alpha^{-})$ is microlocally $0$ in $U$ 
for $t<\delta$ and $t>T-\delta$, provided $\delta$ is small enough.
The proposition then follows directly from (\ref{bzbzb}) with a remainder term $R = {\mathcal O} (h^{\infty})$ in ${\mathcal B} (L^{2} , L^{2})$.
\end{proof}

Now it remains to show that all operators above compose as $h$-FIOs. We shall use several lemmas and we begin with the usual approximation of the quantum propagator.

\begin{Lem}\sl  \label{lem1}
For any $t\in \R$, $e^{-it(P-E)/h}$ is a $h$-FIO of order 0 associated to the canonical relation
\begin{equation*}
\Lambda_{t}=\{ (x,\xi,y,\eta)\in T^*\R^n\times  T^*\R^n; \  (x,\xi)=\exp(tH_{p})(y,\eta)\},
\end{equation*}
uniformly for $t$ in a compact.
\end{Lem}

\begin{proof}
For $t$ small enough, it is well-known that one can write the kernel $K$ of the operator $e^{-it(P-E)/h}$ as
\begin{equation*}
K=\frac{1}{(2\pi h)^n} \int_{\R^n} e^{-i(\varphi(t,x,\theta)-y\cdot \theta +t E)/h} a(t,x,\theta ;h) \, d \theta,
\end{equation*}
modulo an operator ${\mathcal O} (h^{\infty})$ in ${\mathcal B}(L^{2} , L^{2})$ uniformly for $t$ in a compact.
See {\it e.g.} Proposition IV-30 in Robert's book \cite{Ro87_01} or Theorem 10.9 in the book of Evans and Zworski \cite{EvZw07_01}. Here $\varphi$ is a non-degenerate phase function, which  satisfies the eikonal equation
\begin{equation}\label{eiko}
\varphi'_{t}+p(x,\varphi'_{x})=0 ,
\end{equation}
and (see Proposition IV-14 {\it i)} of \cite{Ro87_01})
\begin{equation}   \label{a3}
( x, \varphi'_{x} ) = \exp (t H_{p} ) ( \varphi'_{\theta} , \theta ) .
\end{equation}
This gives the lemma for $t$ small enough. For other values of $t$, Robert uses the following trick. For some $k\in \N$ large enough, one can write
\begin{equation*}
e^{-it(P-E)/h}=\prod_{j=1}^k e^{-it(P-E)/k h}.
\end{equation*}
It is then easy to see that these operators compose as $h$-FIOs, and that the result is associated to $\Lambda_{t}$ and of order 0.
\end{proof}

\begin{Lem}\sl \label{lem2}
Let $\alpha \in C^{\infty}_{0} (T^{*} \R^{n})$ be such that $H_{p}(x, \xi ) \neq 0$ for all $(x, \xi ) \in \supp \alpha \cap p^{-1} (E_{0})$. There exists $\delta>0$ such that, for any $\chi\in C^{\infty}_{0} (]0,\delta [)$, the operator ${\mathcal L} : L^2(\R^n)\to L^2(\R^n)$ defined by
\begin{equation*}
{\mathcal L} =\frac{i}{h}\int \chi (t) e^{-it(P-E)/h} d t \Op ( \alpha ) ,
\end{equation*}
is a $h$-FIO with compactly supported symbol of order $1/2$ associated to the canonical relation $\Lambda_{\alpha, \chi} (E_{0})$ given by
\begin{align*}
\Lambda_{\alpha, \chi} (E_{0}) = \big\{ (x, \xi , y , \eta ) \in T^*\R^n\times T^* & \R^n ; \ p (y , \eta ) =E_{0} , \ (y, \eta ) \in \supp ( \alpha ) +  B ( 0 , \varepsilon ), \\
&\text{and } \exists t \in \supp \chi + ] - \varepsilon , \varepsilon [ ,\  (x,\xi)=\exp (t H_{p})(y,\eta) \big\} ,
\end{align*}
for any $\varepsilon >0$.
\end{Lem}

\begin{remark}\sl \label{a13}
Note that $\Lambda_{\alpha, \chi} (E_{0})$ is not a closed Lagragian submanifold. Nevertheless, there is no point here  since the support of the symbol of the $h$-FIO does not reach the boundary of $\Lambda_{\alpha, \chi} (E_{0})$ for any $\varepsilon >0$. In particular, the parameter $\varepsilon$ plays no role.
It would be natural to write that the canonical relation of this $h$-FIO is $\Lambda (E_{0})$ given by
\begin{equation*}
\Lambda (E_{0}) = \{ (x,\xi,y,\eta)\in T^*\R^n\times T^*\R^n ; \  p(x,\xi) = E_{0} , \  \exists t \in \R, \  (x,\xi)=\exp (t H_{p})(y,\eta) \} .
\end{equation*}
However, since the Hamiltonian flow vanish at $(0,0)$, $\Lambda (E_{0} )$ is not a manifold. Of course, in the non trapping case, there is not such difficulty and $\Lambda_{\alpha, \chi} (E_{0})$ can be replaced by $\Lambda (E_{0})$.
\end{remark}

\begin{proof}
As in Lemma \ref{lem1}, we have, modulo an operator ${\mathcal O} (h^{\infty})$ in ${\mathcal B}(L^{2} , L^{2})$,
\begin{equation*}
K_{{\mathcal L}}=\frac{i}{(2\pi)^n h^{n+1}}\iint \chi(t) e^{i(\varphi(t,x,\theta)-y\cdot\theta +tE_{0} )/h} e^{i t E_{1}} b(t,x, y, \theta ;h ) \, d t \, d \theta ,
\end{equation*}
and we consider $(t,\theta)$ as phase variables. Here, $e^{i t E_{1}} \chi (t) b (t,x, y , \theta ,h) \sim \sum_{j} b_{j} (t,x, y , \theta ) h^{j}$ is a classical symbol of order $0$ and has compact support in $t, x, y, \theta$ with $\Pi_{y, \theta} \supp ( e^{i t E_{1}} \chi b ) \subset \supp ( \alpha )$. We have to show that the function $\Phi : \R^n \times \R^n \times \R^{n+1} \to \R$ given by
\begin{equation*}
\Phi(x,y, (t,\theta) )=\varphi(t,x,\theta)-y\cdot\theta + t E_{0} ,
\end{equation*}
is a non-degenerate phase function. We denote
\begin{align*}
C_{\Phi}=& \{(x,y,t,\theta) \in \supp (\chi b) ; \  \Phi'_{t}(t,\theta,x,y)=0,\  \Phi'_{\theta}(t,\theta,x,y)=0\}  \\
=& \{ (x,y,t,\theta) \in \supp (\chi b) ; \ \varphi'_{t} + E_{0} = 0 , \  \varphi'_{\theta} = y \} ,
\end{align*}
the critical set of the phase $\Phi$ intersected with the support of the symbol. We have to show that at any point $(x,y , t,\theta )$ of $C_{\Phi}$, the matrix 
\begin{equation*}
\left( \begin{array}{c}
d\Phi'_{t} (x,y,t,\theta)\\
d\Phi'_{\theta} (x,y,t,\theta)
\end{array} \right)
=\left( \begin{array}{cccc}
\varphi''_{t,t}&\varphi''_{t,\theta}&\varphi''_{t,x}&0\\
\varphi''_{\theta,t}&\varphi''_{\theta,\theta}&\varphi''_{\theta,x}&- {\rm Id}
\end{array} \right) ,
\end{equation*}
is of maximal rank. The bottom $n$ rows are clearly independent and it is enough to prove that the first line does not vanish on the compact $C_{\Phi}$. Assume that the first line vanishes at some point of $C_{\Phi}$. At this point, $(y, \theta ) \in \supp ( \alpha )$, $\varphi'_{t} + E_{0} =0$ and $\varphi'_{\theta} =y$. Differentiating \eqref{a3} with respect to $t$, we obtain
\begin{equation*}
( 0 , \varphi''_{t,x} ) = H_{p} \big( \exp (t H_{p} ) ( \varphi'_{\theta} , \theta ) \big) + d_{( \varphi'_{\theta} , \theta )} \exp (t H_{p} ) (\varphi''_{t, \theta} , 0) ,
\end{equation*}
and then
\begin{equation*}
H_{p} \big( \exp (t H_{p} ) ( y , \theta ) \big) =( 0 , 0 ).
\end{equation*}
Since $\big( d_{(x, \xi)} \exp (t H_{p}) \big) (H_{p} (x, \xi )) = H_{p} \big( \exp (t H_{p}) (x, \xi ) \big)$, we deduce
\begin{equation} \label{a4}
H_{p} (y , \theta ) = (0,0).
\end{equation}
Moreover, from $\varphi'_{\theta} =y$, \eqref{a3}, the eikonal equation \eqref{eiko} and $\varphi'_{t} + E_{0} =0$ we have
\begin{equation*}
p (y, \theta ) = p ( \varphi'_{\theta} , \theta ) = p (x , \varphi'_{x}) = -\varphi'_{t} =E_{0}.
\end{equation*}
But since $(y, \theta ) \in \supp  ( \alpha )$ and $H_{p}$ does not vanish on $\supp \alpha \cap p^{-1} (E_{0})$, this contradicts \eqref{a4}. Therefore, $\Phi$ is a non-degenerate phase function and ${\mathcal L}$ is an $h$-FIO with compactly supported symbol associated to
\begin{align*}
\Lambda_{\Phi} =& \{ (x, \Phi'_{x} (x, y, t, \theta) , y, - \Phi'_{y} (x, y, t, \theta) ) ; \ (x, y, t, \theta ) \in C_{\Phi} \}   \\
=& \{ ( x, \varphi'_{x} (t, x, \theta ) , y, \theta ); \ \varphi'_{t} + E_{0} = 0, \ \varphi'_{\theta} =y , \ (x,y,t,\theta) \in \supp (\chi b) \} \Subset \Lambda_{\alpha , \chi} (E_{0}) ,
\end{align*}
thanks to the equations \eqref{eiko} and \eqref{a4}. From Definition \ref{a44}, we obtain that the order of this $h$-FIO is $1/2$.
\end{proof}

We are  now able to prove the following

\begin{Lem}\sl \label{a5}
Let $\alpha \in C^{\infty}_{0} (T^{*} \R^{n})$ be such that $H_{p}(x, \xi ) \neq 0$ for all $(x, \xi ) \in \supp \alpha \cap p^{-1} (E_{0})$. For any $\chi\in C^{\infty}_{0}(]0,+\infty[)$, the operator ${\mathcal L}: L^2(\R^n)\to L^2(\R^n)$ defined by
\begin{equation*}
{\mathcal L} =\frac{i}{h}\int \chi (t) e^{-it(P-E)/h} d t \Op ( \alpha ) ,
\end{equation*}
is a $h$-FIO with compactly supported symbol of order $1/2$ associated with the canonical relation $\Lambda_{\alpha , \chi} (E_{0} )$ given by
\begin{align*}
\Lambda_{\alpha, \chi} (E_{0} ) = \{ (x, \xi , y , \eta ) \in T^*\R^n\times & T^*\R^n ; \ p (y , \eta ) =E_{0} , \ (y, \eta ) \in \supp ( \alpha ) +  B ( 0 , \varepsilon ), \\
&\text{and } \exists t \in \supp \chi + ] - \varepsilon , \varepsilon [ ,\  (x,\xi)=\exp (t H_{p})(y,\eta) \} ,
\end{align*}
for any $\varepsilon >0$.
\end{Lem}

Remark \ref{a13} still applies here and one can, formally, replace $\Lambda_{\alpha , \chi} (E_{0})$ by $\Lambda (E_{0})$.

\begin{proof}
For  $\delta>0$ small enough so that Lemma \ref{lem2} applies, we can find $\widetilde{\chi} \in C^{\infty}_{0}(]0,\delta [)$ so that, for some $\nu >0$,
\begin{equation*}
\sum_{k\in \N}\widetilde\chi(y-\nu k)=1.
\end{equation*}
We have, for some $N\in \N$,
\begin{align*}
{\mathcal L} =& \frac{i}{h}\sum_{k\in\N}\int \chi (t) \widetilde{\chi} (t- \nu k)  e^{-it(P-E)/h} d t \Op ( \alpha ) = \frac{i}{h} \sum_{k=0}^{N} \int \chi (t ) \widetilde{\chi} (t- \nu k)  e^{-it(P-E)/h} d t \Op ( \alpha )   \\
=& \frac{i}{h} \sum_{k=0}^{N} e^{-i \nu k(P-E)/h} \circ \int \chi (t+ \nu k)\widetilde{\chi} (t)  e^{-it(P-E)/h} d t \Op ( \alpha ) .
\end{align*}
Using that the operator in Lemma \ref{lem2} is a $h$-FIO with compactly supported symbol and the Egorov theorem, we can find $\beta , \gamma \in C^{\infty}_{0} (T^{*} \R^{n})$ such that
\begin{equation*}
{\mathcal L} = \frac{i}{h} \sum_{k=0}^{N} \Op ( \beta ) e^{-i \nu k(P-E)/h} \Op ( \gamma ) \circ \int \chi (t+ \nu k)\widetilde{\chi} (t)  e^{-it(P-E)/h} d t \Op ( \alpha )  + R,
\end{equation*}
where $R = \CO (h^{\infty} )$ in ${\mathcal B} (L^{2} , L^{2})$. From Lemma \ref{lem2},
\begin{equation*}
\Op ( \beta ) e^{-i \nu k(P-E)/h} \Op ( \gamma ) \in \mathcal{I}_{h}^{0} \big( \R^{n} \times \R^{n} , \Lambda_{k} {}' \big)
\end{equation*}
with compactly supported symbol.

To finish the proof, it is enough to compose this operator with the $h$-FIOs described in Lemma \ref{lem1}.
Since $\Lambda_{k}$ is given by a canonical transformation, $\Lambda_{k} \times \Lambda_{\alpha, \chi (t + \nu k ) \widetilde{\chi} (t)} (E_{0})$ intersects $T^{*} \R^{n} \times \diag ( T^{*} \R^{n} \times T^{*} \R^{n} ) \times T^{*} \R^{n}$ transversally (cleanly with excess $e=0$). Then, using Theorem \ref{composeFIO}, they compose as $h$-FIOs with compactly supported symbol of order $1/2$ with canonical relation
\begin{equation*}
\Lambda_{k} \circ \Lambda_{\alpha, \chi (t + \nu k ) \widetilde{\chi} (t)} (E_{0} ) = \Lambda_{\alpha, \chi (t) \widetilde{\chi} (t - \nu k)} (E _{0}) .
\end{equation*}
Summing over $k$, we obtain the lemma.
\end{proof}

\begin{proof}[Proof of Theorem \ref{ThRes}]
From Proposition \ref{resFIO}, to calculate ${\mathcal I}$, it is enough to compose the $h$-FIOs appearing in \eqref{a6}. We will use Theorem \ref{composeFIO} to make these compositions. As in the end of the proof of Lemma \ref{a5}, we have from Lemma \ref{lem1} and Lemma \ref{a5},
\begin{equation}\label{a7}
\left(\frac{i}{h} \int \chi(t) e^{-it(P-E)/h} dt\right) e^{iT_{1}(P-E)/h} \Op (\alpha^{-}) \in \mathcal{I}_{h}^{\frac{1}{2}} \big( \R^{n} \times \R^{n} , \Lambda_{\alpha \circ \exp (T_{1} H_{p}) , \chi } (E_{0}) ' \big) ,
\end{equation}
with compactly supported symbol.

We recall that, from \cite[Remark 2.7]{BoFuRaZe07_01},
\begin{equation} \label{a8}
\mathcal{J}(E) \in \mathcal{I}_{h}^{-\frac{\sum_{j=1}^{n} \lambda_{j}}{2\lambda_1}} \big( \R^{n} \times \R^{n} , \Lambda_+ \times \Lambda_-  {}' \big) ,
\end{equation}
with compactly supported symbol. The manifold $( \Lambda_{+} \times \Lambda_{-} ) \times \Lambda_{\alpha \circ \exp (T_{1} H_{p}) , \chi } (E_{0})$ intersects $T^{*} \R^{n} \times \diag ( T^{*} \R^{n} \times T^{*} \R^{n} ) \times T^{*} \R^{n}$ cleanly with excess $e=1$ and
\begin{equation*}
( \Lambda_{+} \times \Lambda_{-} ) \circ \Lambda_{\alpha \circ \exp (T_{1} H_{p}) , \chi} (E_{0}) \subset \Lambda_{+} \times \Lambda_{-}.
\end{equation*}
Then, the composition rules for the $h$-FIOs in \eqref{a7} and \eqref{a8} implies that
\begin{equation}\label{a10}
\mathcal{J}(E) \left(\frac{i}{h} \int \chi(t) e^{-it(P-E)/h} dt\right) e^{iT_{1}(P-E)/h} \Op (\alpha^{-}) \in \mathcal{I}_{h}^{1 -\frac{\sum_{j=1}^{n} \lambda_{j}}{2\lambda_1}} \big( \R^{n} \times \R^{n} , \Lambda_{+} \times \Lambda_{-} {}' \big),
\end{equation}
with compactly supported symbol.

Finally, from Lemma \ref{lem1}, 
\begin{equation}\label{a11}
e^{-iT_{1}(P-E)/h} \Op (\alpha^+_{T_{1}}) \in \mathcal{I}_{h}^{0} \big( \R^{n} \times \R^{n} , \Lambda_{T_{1}} {}' \big) ,
\end{equation}
with a compactly supported symbol. Since $\Lambda_{T_{1}}$ is given by a canonical transformation, the intersection between $\Lambda_{T_{1}} \times ( \Lambda_{+} \times \Lambda_{-} )$ and $T^{*} \R^{n} \times \diag ( T^{*} \R^{n} \times T^{*} \R^{n} ) \times T^{*} \R^{n}$ is clean with excess $e =0$. Moreover
\begin{equation*}
\Lambda_{T_{1}} \circ ( \Lambda_{+} \times \Lambda_{-} ) \subset \Lambda_{+} \times \Lambda_{-}.
\end{equation*}
Then, \eqref{a6} and  the composition of the $h$-FIOs appearing in \eqref{a10} and \eqref{a11} gives
\begin{equation}\label{a12}
\Op (\alpha^{+}) \CR (E + i 0) \Op (\alpha^{-}) \in \mathcal{I}_{h}^{1 -\frac{\sum_{j=1}^{n} \lambda_{j}}{2\lambda_1}} \big( \R^{n} \times \R^{n} , \Lambda_{+} \times \Lambda_{-} {}' \big) .
\end{equation}
\end{proof}

\Subsection{Microlocal representation of the spectral function}\label{Smspf}

We give here the representation of the spectral function as an oscillatory 
integral operator microlocally near any point $(\rho_+, \rho_-) \in 
\widetilde{\Lambda_+\times\Lambda_-}.$
The oscillatory integral representation near points in 
$\widetilde{\Lambda_-\times\Lambda_+}$ is analogous.

\begin{Th}\sl
Let $(\rho_+, \rho_-) \in\widetilde{\Lambda_+\times\Lambda_-}$. Then there exist $m\in\mathbb{N}$, a non-degenerate phase function $\Psi\in C^{\infty}\left(\mathbb{R}^{2n+m}\right)$ and a symbol $b\in S_{2n+m}^{1-\frac{\sum_{j=1}^{n}\lambda_{j}}{2\lambda_1}+\frac{n}{2}+\frac{m}{2}}(1)$ such that, microlocally near $\left(\rho_+, \rho_-\right)$,
\begin{equation*}
e_{E} (x, y ;h) = \int_{\R^{m}} e^{i \Psi(x, y, \tau) /h} b(x, y, \tau ;h) \, d \tau .
\end{equation*}

Furthermore, if $(\rho_+, \rho_-) \in \widetilde{\Lambda_+\times\Lambda_-}$ and the projections $\pi: T^{*}\mathbb{R}^{n} \longrightarrow \mathbb{R}^{n}$ are diffeomorphisms when restricted to some neighborhood of $\rho_{\pm}$ in $\Lambda_{\pm}$, then there exists a symbol $b\in S^{1-\frac{\sum_{j=1}^{n}\lambda_{j}}{2\lambda_{1}}+\frac{n}{2}}_{2n}(1)$ such that, microlocally near $\left(\rho_+, \rho_-\right)$,
\begin{equation*}
e_{E} (x, y ;h) = e^{i (S_{+} (x) + S_{-} (y) ) /h} b(x, y ;h) ,
\end{equation*}
where
\begin{equation*}
S_{\pm} ( z) = \int_{\gamma_{\pm} (z)} \frac12 \vert \xi_{\pm}
(t) \vert^{2} + E_{0} -V(x_{\pm}(t)) \, d t ,
\end{equation*}
are the actions over the Hamiltonian half-trajectories $\gamma_{\pm} (z) = (x_{\pm} , \xi_{\pm} )$ which start at $\pi_{\vert_{\Lambda_{\pm}}}^{-1} (z)$ and approach $(0, 0)$ as $t\to \mp\infty$.
\end{Th}

\begin{proof}
The first part of the theorem follows from \cite[Theorem 1]{Al05_02} and Theorem \ref{ThRes}. Assume now that $\pi_{\vert_{\Lambda_{\pm}}}$ is a diffeomorphism in a neighborhood of $\rho_{\pm}$. We will now show that
\begin{equation}\label{a41}
\Lambda_{\pm} = \big\{ ( z, \pm \partial_{z} S_{\pm} ( z) ) ; \  z \text{ near } \pi ( \rho_{\pm}) \big\} ,
\end{equation}
locally near $\rho_{\pm}$. We only prove \eqref{a41} for $\Lambda_{+}$ since the manifold $\Lambda_{-}$ can be treated by the same way. Let
\begin{equation*}
( x_{+} (t, z) , \xi_{+} (t, z) ) = \exp (t {\rm H}_{p} ) \big( \pi_{\vert_{\Lambda_{+}}}^{-1}(z) \big) .
\end{equation*}
From the definition of the Hamiltonian vector field, we have
\begin{align}
\partial_{t} \big( \xi_{+} (t, z ) \partial_{z} ( x_{+} ( t , z) ) \big) =& \xi_{+} (t, z ) \partial_{z} ( \xi_{+} ( t , z ) ) - ( \partial_{x} V) (x_{+} (t, z )) \partial_{z} ( x_{+} ( t , z ) )   \nonumber    \\
=& \frac{1}{2} \partial_{z} \big( \vert \xi_{+} (t, z ) \vert^{2} \big) - \partial_{z} \big( V (x_{+}(t, z )) \big)  \nonumber    \\
=& \partial_{z} \Big( \frac{1}{2} \vert \xi_{+} (t, z ) \vert^{2} + E_{0} -  V (x_{+}(t, z ) ) \Big) .  \label{a42}
\end{align}
Moreover, as $t \to - \infty$, we have $\xi_{+} (t, z) \to 0$ and
\begin{equation*}
\partial_{z} ( x_{+} ( t , z) ) = d \Pi_{x} \circ d \exp ( t H_{p} ) \big( \partial_{z} x_{+} (0, z ) , \partial_{z} \xi_{+} (0, z ) \big) \longrightarrow 0,
\end{equation*}
since $(  x_{+} (0, z ) ,  \xi_{+} (0, z ) ) \in \Lambda_{+}$ for all $z$ and $0$ is a unstable node of $H_{p}$ restricted to $\Lambda_{+}$. Using $x_{+} ( 0 , z ) =z$, we obtain
\begin{equation*}
\partial_{z} S_{\pm} ( z) = \int_{- \infty}^{0} \partial_{z} \Big( \frac{1}{2} \vert \xi_{+} (t, z ) \vert^{2} + E_{0} -  V (x_{+}(t, z )) \Big) d s =\xi_{+} (0, z ) \partial_{z} ( x_{+} ( 0 , z ) ) = \xi_{+} (0, z ) .
\end{equation*}
Since $\Lambda_{+} = \{ ( z , \xi_{+} (0, z ) ) ; \  z \text{ near } \pi ( \rho_{\pm}) \}$ locally near $\rho_{+}$, we get \eqref{a41}. Then the second part of the theorem follows again from \cite[Theorem 1]{Al05_02} and Theorem \ref{ThRes}.
\end{proof}

\begin{remark}\sl
From \cite[Section 2.2]{BoFuRaZe07_01} we have that there exists a neighborhood $\Omega\subset T^{*}\mathbb{R}^{n}$ of $(0, 0)$ such that the projection $\pi: T^{*}\mathbb{R}^{n}\to\mathbb{R}^{n}$ restricted to $\Omega\cap\Lambda_{\pm}$ is a diffeomorphism.
\end{remark}

\section{The scattering matrix}\label{Ssa}

\Subsection{Representation of the scattering matrix}\label{Srepr}

Here we review the representation of the short range scattering 
matrix
which we shall use in the proof of Theorem \ref{Thsa}.
The construction is close to the one used by Robert and Tamura 
\cite{RoTa89_01} and constitutes a semi-classical adaptation of the 
representation of the short range amplitude originally established by Isozaki and 
Kitada \cite{IsKi86_01}. Their starting point is 
a set of  WKB parametrices for the wave operators given in (\ref{waveops}).

For $R_{0}\gg 0,$ $1<d_4<d_3<d_2<d_1<d_0,$ and 
$0<\sigma_4<\sigma_3<\sigma_2<\sigma_1<\sigma_0<1$ Robert and 
Tamura construct phase 
functions $\Phi_{\pm}$ and symbols $\left(a_{\pm j}\right)_{j=0}^{\infty}$ 
and  $\left(b_{\pm j}\right)_{j=0}^{\infty}$ such that: 
\begin{enumerate}[{\it i)}]
\item $\Phi_{\pm}\in C^{\infty}(T^{*}\mathbb{R}^{n})$ solve the eikonal 
equation 
\begin{equation} \label{a14}
\frac{1}{2} \vert\nabla_{x}\Phi_{\pm}(x, \xi)\vert^{2}+V(x)=\frac{1}{2} \xi^{2}
\end{equation}
for $(x, \xi)\in\Gamma_{\pm}(R_0, d_0, \pm\sigma_0)$  respectively (see (\ref{gammainout})
for the definition of these sets).

\item Let $A_{m}(\Omega)$ be the class of symbols $a$ such that $(x, \xi)\mapsto a(x, \xi ;h)$ belongs to $C^{\infty}(\Omega)$ and, for any $( \alpha , \beta ) \in \mathbb{N}^{n} \times \mathbb{N}^{n}$ and $L>0$,
\begin{equation*}
\vert \partial_{x}^{\alpha}\partial_{\xi}^{\beta}a(x, \xi ;h) \vert \leq C_{\alpha , \beta} \langle x \rangle^{m-|\alpha|}\langle\xi\rangle^{-L},
\end{equation*}
for all $(x, \xi ) \in \Omega$. We have, from Proposition 2.4 of \cite{IsKi85_01},
\begin{equation} \label{a17}
\Phi_{\pm}(x,\xi)- \langle x, \xi \rangle \in A_{1 - \rho} \left(\Gamma_{\pm}(R_0, d_0, \pm\sigma_{0})\right).
\end{equation}

\item For all $(x, \xi)\in T^{*}\mathbb{R}^{n}$
\begin{equation*}
\Big\vert  \frac{\partial^{2}\Phi_{\pm}}{\partial_{x_{j}}\partial_{\xi_{k}}}(x, \xi)-\delta_{jk}
\Big\vert
<\varepsilon (R_0), 
\end{equation*}
where $\delta_{jk}$ is the Kronecker delta and $\varepsilon(R_0)\to 0$ as $R_{0}\to\infty.$

\item $\left( a_{\pm j} \right)_{j}$ and $\left( b_{\pm j} 
\right)_{j}$ are determined inductively 
as solutions to certain transport equations and satisfy 
$$
a_{\pm j}\in 
A_{-j}(\Gamma_{\pm}(3R_0, d_1, \pm \sigma_1)),\quad 
\supp a_{\pm j}\subset\Gamma_{\pm}(3R_0, 
d_1, \pm \sigma_1), 
$$
$$
b_{\pm j}\in A_{-j}(\Gamma_{\pm}(5R_0, d_3, \pm\sigma_{4})),\quad 
\supp b_{\pm j}\subset\Gamma_{\pm}(5R_0, 
d_3, \pm\sigma_{4}).
$$
\end{enumerate}
Using the Borel process, we can find two symbols $a_{\pm}\in 
A_{0}(\Gamma_{\pm}(3R_0, d_1, 
\pm \sigma_1))$ and 
$b_{\pm}\in A_{0}(\Gamma_{\pm}(5R_0, d_3, \pm\sigma_{4}))$ such 
that 
$a_{\pm}\sim\sum_{j=0}^{\infty} h^j a_{\pm j}$ and 
$b_{\pm}\sim\sum_{j=0}^{\infty} h^j b_{\pm j}$.

For a symbol $c$ and a phase function $\varphi$, we denote by $I_{h}(c, \varphi)$ the oscillatory integral 
\begin{equation*}
I_{h}(c, \varphi)=\frac{1}{(2\pi h)^{n}}\int_{\mathbb{R}^{n}} e^{i(\varphi(x, \xi)-\langle y, 
\xi\rangle)/h}c(x, \xi ;h) \, d \xi
\end{equation*}
and let 
\begin{equation*}
\begin{aligned}
K_{\pm a}(h) & =P(h)I_{h}(a_{\pm}, \Phi_{\pm})-I_{h}(a_{\pm}, \Phi_{\pm})P_{0}(h)\\
K_{\pm b}(h) & =P(h)I_{h}(b_{\pm}, \Phi_{\pm})-I_{h}(b_{\pm}, \Phi_{\pm})P_{0}(h).
\end{aligned}
\end{equation*} 
The scattering matrix, ot more precisely the operator  $\CT(E, h)$ is then given for 
$E\in\big] \frac{2}{d_{4}^{2}}, \frac{d_{4}^{2}}{2} \big[$ by (see \cite[Theorem 3.3]{IsKi86_01})
\begin{equation}\label{ik}
\CT(E, h)= T_{+1}(E,h) + T_{-1}(E, h) - T_{2}(E, h) ,
\end{equation}
where
\begin{equation*}
T_{\pm 1}(E, h)=F_{0}(E, h)I_{h}(a_{\pm}, \Phi_{\pm})^{*}K_{\pm b}(h)
F_{0}^{-1}(E, h)
\end{equation*}
and, with $F_{0}(E, h)$ given in \eqref{pmlo1},
\begin{equation}\label{t2}
T_{2}(E, h)=F_{0}(E, 
h)K_{+a}^{*}(h)\CR(E+i0, 
h)\left(K_{+b}(h)+K_{-b}(h)\right)F_{0}^{*}(E, h).
\end{equation}

\Subsection{Proof of Theorem \ref{Thsa}}\label{Sthsa}

Since $\CS(E, h)$ is a unitary operator on $L^{2}(\mathbb{S}^{n-1}),$ 
we have, by \cite[Lemma 1]{Al06_01}, that its kernel $K_{\CS(E, h)}\in \CS'_{h}(\mathbb{S}^{n-1}\times\mathbb{S}^{n-1})$ and 
therefore $K_{\CT(E, h)}\in \CS'_{h}(\mathbb{S}^{n-1}\times\mathbb{S}^{n-1}).$

Since we are working away from the diagonal in $\mathbb{S}^{n-1}\times\mathbb{S}^{n-1}$ 
we can use integration by parts, as in \cite{RoTa89_01} and \cite{Mi04_01}, to obtain
\begin{equation*}
K_{T_{\pm 1}(E,h)}=\mathcal{O}_{C^{\infty}(\mathbb{S}^{n-1}\times\mathbb{S}^{n-1}\backslash\diag 
(\mathbb{S}^{n-1}\times\mathbb{S}^{n-1}))}(h^{\infty}).
\end{equation*}
Therefore
\begin{equation}\label{Tpm1}
WF_{h}^{f} \big( {K_{T_{\pm 1}(E, h)}}_{|_{\mathbb{S}^{n-1}\times\mathbb{S}^{n-1}\setminus \diag \left(\mathbb{S}^{n-1}\times\mathbb{S}^{n-1}\right)}} \big) = \emptyset.
\end{equation}

We now observe that the proof of \cite[Lemma 2.1]{RoTa89_01} depends only on the estimate 
\eqref{resest} and the support properties of the symbols $a_{\pm}$ and $b_{\pm}$, and by the 
same method of proof, we obtain the following strengthened version of \cite[Lemma 2.1]{RoTa89_01}.

\begin{Lem}\label{l2.1RT}\sl Suppose {\it (A1)} and {\it (A2)}.
For $\gamma\gg 1$, 
\begin{enumerate}[i)]
\item $\Vert K_{+a}^{*}(h)\CR(E+i0)K_{+b}(h) \Vert_{\mathcal{B} ( L_{-\gamma}^{2} , L_{\gamma}^{2} )} = \mathcal{O}(h^{\infty})$ ,
\item $\Vert K_{+a}^{*}(h)\CR(E+i0)\left(1-\chi_b\right)K_{+b}(h) \Vert_{\mathcal{B} ( L_{-\gamma}^{2} , L_{\gamma}^{2} )} = \mathcal{O}(h^{\infty})$ ,
\item $\Vert \left(\left(1-\chi_a\right)K_{+a}(h)\right)^{*}\CR(E+i0)\chi_b K_{-b}(h)\Vert_{\mathcal{B} ( L_{-\gamma}^{2} , L_{\gamma}^{2} )} = \mathcal{O}(h^{\infty})$ .
\end{enumerate}
\end{Lem}

From \eqref{Tpm1}, Lemma \ref{l2.1RT}, and \cite[Equation (10)]{Al06_01} we then 
conclude, as in \cite[Corollary, page 168]{RoTa89_01}, that
\begin{equation}\label{Gu76_01}
WF_{h}^{f}(\chi(K_{\CS(E,h)}-c_{1} \, G))=\emptyset,
\end{equation}
for every $\chi\in 
C^{\infty}\left(\mathbb{S}^{n-1}\times\mathbb{S}^{n-1}\backslash
\diag(\mathbb{S}^{n-1}\times\mathbb{S}^{n-1})\right),$ where
\begin{equation}\label{sared}
G(\theta, \omega ; E, h) =\bra \CR(E+i0) e^{i\Phi_{-} ( y, \sqrt{2E_{0}}\omega ) /h}g_{-}( y, \omega ; h) , \; 
e^{i \Phi_{+} ( x, \sqrt{2E_{0}} \theta ) /h} g_{+}( x, \theta ; h)\ket ,
\end{equation}
\begin{equation*}\label{defg+}
g_{+}(x, \theta; h) = e^{- i \Phi_{+} (x, \sqrt{2E_{0}}\theta )/h}
[\chi_{a}, P_0(h)]a_{+} \big(x, \sqrt{2E} \theta ; h \big)
e^{i \Phi_{+} (x, \sqrt{2E}\theta )/h},
\end{equation*}
\begin{equation*}\label{defg-}
g_{-}(y, \omega; h)=e^{-i \Phi_{-} (y, \sqrt{2E_{0}}\omega )/h} [\chi_{b}, P_0(h)] b_{-} \big( y, \sqrt{2E} \omega ; h \big) e^{i \Phi_{-} (y, \sqrt{2E}\omega ) /h},
\end{equation*}
and 
\begin{equation*}
c_{1} = c_{1}(n, E, h)=-2i\pi(2E)^{\frac{n}{2} - 1}(2\pi h)^{-{n}}.
\end{equation*}
Here $\chi_{a} (x)$ and $\chi_{b} (y)$ are $C^{\infty}_{0} (\R^{n})$ functions with value $1$ in a large disc. In particular, the symbols $g_{+} (x, \theta ;h)$, $g_{-}(y, \omega; h) \in S^{-1} (1)$ have compact support (uniformly with respect to $h$). Notice that we have used the fact that $E-E_0=E_1h$.

From \eqref{sared}, one can see that $G ( \theta , \omega ;E ,h)$ is the kernel of the operator
\begin{equation}  \label{a24}
{\mathcal G} = {\mathcal M}_{+}^{*} \CR(E+i0) {\mathcal M}_{-},
\end{equation}
where ${\mathcal M}_{\pm} : L^2(\S^{n-1}) \to L^2(\R^{n})$ are given by
\begin{gather*}
K_{{\mathcal M}_{+}} (x, \theta ;h ) = e^{i \Phi_{+} ( x, \sqrt{2E_{0}} \theta ) /h} g_{+} ( x, \theta ; h) ,  \\
K_{{\mathcal M}_{-}} (x, \omega ;h ) = e^{i \Phi_{-} ( y, \sqrt{2E_{0}} \omega ) /h} g_{-} ( y, \omega ; h) .
\end{gather*}
The operator ${\mathcal M}_{+}$ can be view as an $h$-FIO
\begin{equation} \label{a27}
{\mathcal M}_{+} \in \mathcal{I}_{h}^{-\frac{2 n +3}{4}} \big( \R^{n} \times \S^{n-1} , C_{+} {}' \big) ,
\end{equation}
with compactly supported symbol (and no phase variable). The canonical relation $C_{+}$ is given by
\begin{equation} \label{a23}
\begin{aligned}
C_{+} = \big\{ (x , \xi , \theta , \sqrt{2 E_{0}} & z_{+} ) ; \ \xi = \partial_{x} \Phi_{+} (x, \sqrt{2 E_{0}} \theta ),  \\
&\sqrt{2 E_{0}} z_{+} = - \partial_{\theta} \Phi_{+} (x, \sqrt{2 E_{0}} \theta ) , \ (x, \theta ) \in \supp ( g_{+} ) + B (0, \varepsilon ) \big\} ,
\end{aligned}
\end{equation}
for any $\varepsilon >0$ (see Remark \ref{a13}). Notice that $\partial_{\theta}$ denotes the derivative on $\S^{n-1}$. Now we calculate more precisely $C_{+}$.

\begin{Lem}\sl \label{a31}
We have
\begin{equation*}
C_{+} = \big\{ (x , \xi , \theta , - \sqrt{2 E_{0}} z_{+}) ; \ \exists t \in \R ,  \ (x, \xi ) = \gamma_{+} (t , z_{+} , \theta , E_{0} ), \ (x, \theta ) \in \supp ( g_{+} ) + B (0, \varepsilon ) \big\} ,
\end{equation*}
where $\gamma_{+} (t , z , \alpha , E )$ is defined in \eqref{a19}.
\end{Lem}

\begin{proof} We set 
\begin{equation*}
\Psi_{+} (x, \theta ) = \Phi_{+} (x,\sqrt{2E_0} \theta ).
\end{equation*}
Let $x$ be such that $(x, \sqrt{2 E_0} \theta ) \in \Gamma_{+} (3R_0, d_1, \sigma_1)$. We denote
\begin{equation} \label{a20}
( y (t,x, \theta ) , \eta (t, x, \theta ) ) = \exp (t H_{p} ) ( x , \partial_{x} \Psi_{+} (x, \theta ) ) .
\end{equation}
Remark that $(y (t,x, \theta ) , \sqrt{2 E_{0}} \theta )$ stays in $\Gamma_{\pm}(R_{0}, d_{0},  \sigma_{0} )$ for all $t \geq 0$ and then the following limits exist
\begin{equation} \label{a18}
\left\{ \begin{aligned}
&\lim_{t \to + \infty} \eta (t,x, \theta ) = \eta_{\infty} \in \sqrt{2 E_{0}} \S^{n-1}   \\
&\lim_{t \to + \infty}  y  (t,x, \theta ) - t \eta_{\infty} = y_{\infty} .
\end{aligned} \right.
\end{equation}
By \eqref{a14} we have
\begin{equation}\label{trucmachin}
\frac{1}{2} \vert\partial_{x}\Psi_{\pm}(x,  \theta)\vert^{2}+V(x)=\frac{1}{2} E_0.
\end{equation}
Differentiating with respect to $x$ we obtain
\begin{equation*}
( \partial_{x, x}^{2} \Psi _+) (x,  \theta ) ( \partial_{x} \Psi_+ ) (x, \theta ) + ( \partial_{x} V ) (x)= 0.
\end{equation*}
Therefore, the Hamiltonian flow $H_{p}$ is tangent to $\{ (x, \partial_{x} \Psi_+(x, \theta) ) ; \ x \in \R^{n} \}$ and then
\begin{equation} \label{a15}
\eta (t, x, \theta ) = ( \partial_{x} \Psi_{+} ) (y (t,x, \theta ) ,  \theta ),
\end{equation}
for all $t \geq 0$. In particular, from \eqref{a17},
\begin{align}
\eta_{\infty} =& \lim_{t \to + \infty} \eta (t, x, \theta ) = \lim_{t \to + \infty} \sqrt{2 E_{0}} \theta + {\mathcal O} \big( \vert y (t,x, \theta ) \vert^{-\rho} \big)  \nonumber \\
=& \sqrt{2 E_{0}} \theta . \label{a21}
\end{align}

On the other hand, differentiating \eqref{a14} with respect to $\theta$, we get
\begin{equation} \label{a16}
( \partial_{x, \theta}^{2} \Psi _+) (x,  \theta ) ( \partial_{x} \Psi_+ ) (x,  \theta )= 0.
\end{equation}
Using \eqref{a15} and \eqref{a16}, we obtain
\begin{align*}
\partial_{t} ( \partial_{\theta} \Psi_+ ) ( y (t,x, \theta ) , \theta ) =& ( \partial_{x, \theta}^{2} \Psi_+) ( y (t,x, \theta ) , \theta ) \partial_{t} y (t,x, \theta )  \\
=& ( \partial_{x, \theta}^{2} \Psi_+ ) ( y (t,x, \theta ) ,  \theta ) \eta (t,x, \theta )   \\
=& ( \partial_{x, \theta}^{2} \Psi_+ ) ( y (t,x, \theta ) , \theta ) ( \partial_{x} \Psi_{+}) (y (t,x, \theta ) ,  \theta ) \\
=& 0.
\end{align*}
Now \eqref{a17} and \eqref{a18} yield
\begin{align}
( \partial_{\theta} \Psi_+ ) ( x , \theta ) =&  \lim_{t \to + \infty} ( \partial_{\theta} \Psi_+) ( y (t,x, \theta ) , \theta )    \nonumber   \\
=& \lim_{t \to + \infty} \sqrt{2 E_{0}} \Pi_{\theta^{\perp}} y (t,x, \theta ) + {\mathcal O} \big( \vert y (t,x, \theta ) \vert^{1-\rho} \big)   \nonumber   \\
=& \sqrt{2 E_{0}} \Pi_{\theta^{\perp}} y_{\infty} , \label{a22}
\end{align}
where $\Pi_{\theta^{\perp}}$ is the orthogonal projection on the hyperplane orthogonal to $\theta$:
\begin{equation*}
\Pi_{\theta^{\perp}} x = x - \<x, \theta \> .
\end{equation*}

Finally, let $(x, \xi , \theta , - \sqrt{2 E_{0}} z_{+} ) \in C_{+}$. The asymptotic momentum and position  \eqref{a18} of the Hamiltonian curve \eqref{a20} were calculated in \eqref{a21} and \eqref{a22}. Then, there exist $t \in \R$ such that
\begin{equation*}
( x , \nabla_{x} \Phi_{+} (x, \sqrt{2 E_{0}} \theta ) ) = \gamma_{+} ( t , \theta ,\Pi_{\theta^{\perp}} y_{\infty} , E_{0} ) ,
\end{equation*}
and, from \eqref{a23}, we conclude
\begin{equation*}
( x , \xi ) = \gamma_{+} ( t , \theta ,z_{+} ,  E_{0} ) .
\end{equation*}
\end{proof}

The same way,
\begin{equation} \label{a28}
{\mathcal M}_{-} \in \mathcal{I}_{h}^{-\frac{2 n +3}{4}} \big( \R^{n} \times \S^{n-1} , C_{-} {}' \big) ,
\end{equation}
with compactly supported symbol (and no phase variable). The canonical relation $C_{-}$ is given by
\begin{equation*}
C_{-} = \big\{ (y , \eta , \omega , - \sqrt{2 E_{0}} z_{-}) ; \ \exists t \in \R ,  \ (y , \eta ) = \gamma_{-} (t , z_{-} , \omega , E_{0} ), \ (y, \omega ) \in \supp ( g_{-} ) + B (0, \varepsilon ) \big\} .
\end{equation*}

Let now $( \theta^{0} , z_{+}^{0} , \omega^{0} , z_{-}^{0} ) \in \widetilde{\Lambda_{+}^{\infty} \times \Lambda_{-}^{\infty}}$ be as in Theorem \ref{Thsa}. Let $\beta_{\pm} \in C_{0}^{\infty} ( T^{*} \S^{n-1} )$ with $\beta_{+}$ (resp. $\beta_{-}$) supported in a small neighborhood of $( \theta^{0} , z_{+}^{0} )$ (resp. $( \omega^{0} , z_{-}^{0} )$) and equal to $1$ near $( \theta^{0} , z_{+}^{0} )$ (resp. $( \omega^{0} , z_{-}^{0} )$). From \eqref{Gu76_01} and \eqref{a24}, we have
\begin{equation*}
\Op ( \beta _{+} ) \CS (E,h) \Op ( \beta _{-} ) = c_{1}\Op ( \beta _{+} ) {\mathcal M}_{+}^{*} \CR(E+i0) {\mathcal M}_{-} \Op ( \beta _{-} ) +R,
\end{equation*}
where $R = {\mathcal O}(h^{\infty})$ in ${\mathcal B} (L^{2} , L^{2} )$. Let now $\alpha_{\pm} \in C^{\infty}_{0} (T^{*} \R^{n} )$ supported near
\begin{equation*}
N_{\pm} = C_{\pm} \circ \supp ( \beta_{\pm}) \bigcap \big( \Pi_{x} \supp g_{\pm} \times \R^{n} \big) ,
\end{equation*}
and equal to $1$ near this set. Then, the composition rules for $h$-FIOs implies
\begin{equation}  \label{a25}
\Op ( \beta _{+} ) \CS (E,h) \Op ( \beta _{-} ) =c_{1}\Op ( \beta _{+} ) {\mathcal M}_{+}^{*} \Op ( \alpha _{+} ) \CR(E+i0) \Op ( \alpha _{-} ) {\mathcal M}_{-} \Op ( \beta _{-} ) +R,
\end{equation}
where $R = {\mathcal O}(h^{\infty})$ in ${\mathcal B} (L^{2} , L^{2} )$.

Note that $N_{+}$ is arbitrary close to
\begin{equation*}
N_{+}^{0} = \{ \gamma_{+} ( t , \theta^{0},z_{+}^{0} , E_{0} ) ; \ t \in \R \} \bigcap \big( \Pi_{x} \supp g_{+} \times \R^{n} \big) ,
\end{equation*}
and $N_{-}$ is arbitrary close to
\begin{equation*}
N_{-}^{0} = \{ \gamma_{-} ( t , \omega^{0} , z_{-}^{0} , E_{0} ) ; \ t \in \R \} \bigcap \big( \Pi_{x} \supp g_{-} \times \R^{n} \big) .
\end{equation*}
Every $( \rho_{+} , \rho_{-} ) \in N^{0}_{+} \times N^{0}_{-}$ is in $\widetilde{\Lambda_{+} \times \Lambda_{-}}$ because $( \theta^{0} , z_{+}^{0} , \omega^{0} , z_{-}^{0} ) \in \widetilde{\Lambda_{+}^{\infty} \times \Lambda_{-}^{\infty}}$. Up to a finite summation in $( \rho_{+} , \rho_{-} )$ since $\Pi_{x} \supp g_{\pm}$ is compact, we can assume that $\alpha_{\pm}$ is localized in a small neighborhood of such $( \rho_{+} , \rho_{-} )$. To prove the theorem, we will compose the $h$-FIOs appearing in the formula \eqref{a25}.

The manifold $( \Lambda_{+} \times \Lambda_{-} ) \times C_{-}$ intersects $T^{*} \R^{n} \times \diag ( T^{*} \R^{n} \times T^{*} \R^{n} ) \times T^{*} \S^{n-1}$ cleanly with excess $e=1$ and
\begin{equation*}
( \Lambda_{+} \times \Lambda_{-} ) \circ C_{-} \subset \Lambda_{+} \times \Lambda_{-}^{\infty} .
\end{equation*}
Then the composition rules between the $h$-FIOs 
\begin{equation*}
\Op ( \alpha _{+} ) \CR(E+i0) \Op ( \alpha _{-} ) \in \CI_h^{1-\frac{\sum_{j=1}^{n}\lambda_{j}}{2\lambda_1}} \big( \R^{n} \times \R^{n} , \Lambda_+\times \Lambda_- {}' \big),
\end{equation*}
with compactly supported symbol (see  Theorem \ref{ThRes}), and
\begin{equation*}
{\mathcal M}_{-} \Op ( \beta _{-} ) \in \mathcal{I}_{h}^{-\frac{2 n +3}{4}} \big( \R^{n} \times \S^{n-1} , C_{-} {}' \big) ,
\end{equation*}
with compactly supported symbol (see \eqref{a28}), gives
\begin{equation} \label{a30}
\Op ( \alpha _{+} ) \CR(E+i0) \Op ( \alpha _{-} ) {\mathcal M}_{-} \Op ( \beta _{-} ) \in \mathcal{I}_{h}^{\frac{3 - 2 n}{4} - \frac{\sum_{j=1}^{n}\lambda_{j}}{2\lambda_1}} \big( \R^{n} \times \S^{n-1} , \Lambda_{+} \times \Lambda_{-}^{\infty} {}' \big) ,
\end{equation}
with compactly supported symbol.

But now, taking the adjoint of \eqref{a27}, we obtain 
\begin{equation} \label{a29}
\Op ( \beta _{+} ) {\mathcal M}_{+}^{*} \in \mathcal{I}_{h}^{-\frac{2 n +3}{4}} \big( \S^{n-1}  \times \R^{n} , C_{+}^{-1} {}' \big) ,
\end{equation}
with compactly supported symbol. Here
\begin{equation*}
C_{+}^{-1} = \big\{ ( \theta , z , x , \xi ) ; \  (x, \xi , \theta , z ) \in C_{+} \big\}.
\end{equation*}
The manifold $C_{+}^{-1} \times ( \Lambda_{+} \times \Lambda_{-}^{\infty} )$ intersects $T^{*} \S^{n-1} \times \diag ( T^{*} \R^{n} \times T^{*} \R^{n} ) \times T^{*} \S^{n-1}$ cleanly with excess $e=1$ and
\begin{equation*}
C_{+}^{-1} \circ ( \Lambda_{+} \times \Lambda_{-}^{\infty} ) \subset \Lambda_{+}^{\infty} \times \Lambda_{-}^{\infty} .
\end{equation*}
Then \eqref{a25} and the composition rules between the $h$-FIOs \eqref{a30} and \eqref{a29} gives
\begin{equation}
\Op ( \beta _{+} ) \CS (E,h) \Op ( \beta _{-} ) \in \mathcal{I}_{h}^{\frac{1}{2} - \frac{\sum_{j=1}^{n}\lambda_{j}}{2\lambda_1}} \big( \S^{n-1} \times \S^{n-1} , \Lambda_{+}^{\infty} \times \Lambda_{-}^{\infty} {}' \big) ,
\end{equation}
and this statement is Theorem \ref{Thsa}.

\Subsection{Proof of Theorem \ref{nontrappedcase}}
\label{secntc}

We explain briefly  how to obtain from the preceding arguments the structure of the scattering matrix given in Theorem \ref{nontrappedcase}. It is clear that \eqref{a25} holds also in the present case, and we have first to analyze the structure of the resolvent $\CR(E+i0)$, or more precisely that of
\begin{equation*}
\Op ( \alpha _{+} ) \CR(E+i0) \Op ( \alpha _{-} ),
\end{equation*}
where $\alpha _{\pm}\in C^\infty_{0}(T^*\R^n)$ are now microlocally supported respectively near $\rho_{-}\in p^{-1}(E_{0})$ and $\rho_{+}=\exp (TH_{p})(\rho_{-})$ for some given $T$.

As in Proposition \ref{resFIO}, one can see that 
\begin{equation*}
\Op ( \alpha _{+} ) \CR(E+i0) \Op ( \alpha _{-} )=\Op ( \alpha _{+} )\intÊ\chi(t)e^{-it(P-E)/h} d t \Op ( \alpha_{-} ) +R,
\end{equation*}
with $\Vert R\Vert_{\CB(L^2,L^2)}=\CO(h^\infty)$, for some $\chi\in C^\infty_{0}(]0,2T[)$.
From Lemma \ref{lem2} (see also Remark \ref{a13}), we then know that
\begin{equation}\label{a51}
\Op ( \alpha _{+} ) \CR(E+i0) \Op ( \alpha _{-} )\in \CI_{h}^{1/2}(\R^{n} \times \R^{n} ,\Lambda (E_{0})'),
\end{equation}
where
\begin{equation*}
\Lambda (E_{0}) =\{(x,\xi,y,\eta)\in T^*\R^n\times T^*\R^n ; \  p(x,\xi) = E_{0} , \  \exists t \in \R, \  (x,\xi)=\exp (t H_{p})(y,\eta)\}.
\end{equation*}
The scattering matrix is given by \eqref{a25}. Proceeding as in the previous section and using the fact that
\begin{equation*}
C_{+}^{-1}\circ\Lambda(E_{0})\circ C_{-} \subset \CS\CR(E_{0}) ,
\end{equation*}
we obtain the theorem.

\section{Microlocal representation of the scattering amplitude}\label{Smsa}

Here we discuss the representation of the scattering amplitude as an 
oscillatory integral implied by Theorem \ref{Thsa}.
We also show that this leads to an integral kernel representation of 
the scattering amplitude.

For $\alpha\in\mathbb{S}^{n-1}$ we define the Lagrangian submanifolds 
$\Lambda_{\alpha}^{\pm}\subset T^{*}\mathbb{R}^{n}$ by
\begin{equation*}
\Lambda_{\alpha}^{\pm}= \big\{ \rho\in T^{*}\mathbb{R}^{n}; \ \lim_{t\to\pm\infty} \xi(t, \rho) = \sqrt{2E_0} \alpha \big\} ,
\end{equation*}
and the (modified) actions $S_{\pm}$ over the trajectories
$\gamma_{\pm}=\left(x_{\pm}, \xi_{\pm}\right)\subset\Lambda_{\pm}$ as
\begin{equation}\label{ma}
S_{\pm}=\int_{-\infty}^{\infty} \vert \xi_{\pm}(t) \vert^2-2E_0 
1_{\pm t>0} d t .
\end{equation}
We now have the following

\begin{Lem}\sl \label{ls}
Let $\omega_0, \theta_0\in\mathbb{S}^{n-1}$ be such that $\Lambda_{\omega_{0}}^{-}$ intersects $\Lambda_-$ transversally in $p^{-1} (E_{0})$ and $\Lambda_{\theta_{0}}^{+}$ intersects $\Lambda_+$ transversally in $p^{-1} (E_{0})$. Then
\begin{enumerate}[i)]
\item there exist open sets $O^{\pm}\subset\mathbb{S}^{n-1}$ with $\omega_0\in O^{-}$ and $\theta_0\in O^{+}$ 
such that for every $\omega\in O^{-}$ and every $\theta\in O^{+}$ the intersections of $\Lambda_-$ with 
$\Lambda_{\omega}^{-}$ and of $\Lambda_+$ with $\Lambda_{\theta}^{+}$ are transverse in $p^{-1} (E_{0})$. \label{a37}

\item there exist numbers $N_{\pm}\in\mathbb{N}$ such that for every $\omega\in O^-$ there are exactly 
$N_-$ trajectories $\gamma_{-}^{k} ( \omega )$ in $\Lambda_-$ with initial direction $\omega$ and for every $\theta\in O^+$ there are exactly $N_+$ trajectories $\gamma_{+}^{\ell} ( \theta )$ in $\Lambda_{+}$ with final direction $\theta$. \label{a38}

\item for $k\in\{1, \dots, N_{-} \}$ and $\ell \in\{1, \dots, N_{+} \}$, let $z_{-}^{k} ( \omega )$ and $z_{+}^{\ell} ( \theta )$ be the impact parameters of the curves $\gamma_{-}^{k} ( \omega )$ and $\gamma_{+}^{\ell} ( \theta )$ defined in \eqref{a49}. Then $\omega \to z_{-}^{k} ( \omega )$ and $\theta \to z_{+}^{\ell} ( \theta )$ are $C^{\infty}$ functions in $O^{\pm}$. \label{a39}
\end{enumerate}
\end{Lem}

We can now define the open sets 
\begin{align*}
&\Lambda_{S_{-}^{k}}= \big\{ \big( \omega, - \sqrt{2E_{0}} z_{-}^{k} ( \omega ) \big) \in T^*\S^{n-1};\ \omega \in O^{-} \big\} \subset \Lambda_{-}^{\infty} , \\
&\Lambda_{S_{+}^{\ell}} = \big\{ \big( \theta, - \sqrt{2E_{0}} z_{+}^{\ell} ( \theta ) \big) \in T^*\S^{n-1} ; \ \theta \in O^{+} \big\} \subset \Lambda_{+}^{\infty} .
\end{align*}
Of course, the restrictions to $\Lambda_{S_{-}^{k}}$ and to $\Lambda_{S_{+}^{\ell}}$ of the projection $\pi: T^{*}\mathbb{S}^{n-1} \to \mathbb{S}^{n-1}$ are diffeomorphisms.

\begin{proof}
Let $\rho_{0} \in \Lambda_{+} \cap \Lambda^{+}_{\theta_{0}}$. Then, there exists a $C^{\infty}$ function $f : p^{-1}(E_{0}) \to \R^{n-1}$ defined locally near $\rho_{0}$ such that, for $\rho$ near $\rho_{0}$,
\begin{equation*}
\rho \in \Lambda_{+} \Longleftrightarrow f ( \rho ) =0 ,
\end{equation*}
and the differential of $f$ is of maximal rank.
The same way, since $\Lambda^{+}_{\theta}$ depend smoothly on $\theta$, there exists a $C^{\infty}$ functions $g : p^{-1}(E_{0}) \times \S^{n-1} \to \R^{n-1}$ such that
\begin{equation*}
\rho \in \Lambda^{+}_{\theta} \Longleftrightarrow g ( \rho , \theta ) =0 .
\end{equation*}
and the differential, with respect to $\rho$, of $g$ is of maximal rank.

Now we define
\begin{equation*}
\begin{aligned}
F : \\
{}^{}
\end{aligned}
\left\{ \begin{aligned}
&p ^{-1} (E_{0}) \times \S^{n-1} \longrightarrow  &&\R^{2n-2} \\
&( \rho , \theta ) &&(f (\rho )  , g ( \rho , \theta ) ) 
\end{aligned} \right.
\end{equation*}
and we note that
\begin{equation*}
\rho \in \Lambda_{+} \cap \Lambda^{+}_{\theta} \Longleftrightarrow F ( \rho , \theta ) =0 .
\end{equation*}
Since the intersection $\Lambda_{+} \cap \Lambda^{+}_{\theta_{0}}$ is transverse, the differential of $F$, with respect to $\rho$, is of maximal rank for $\theta = \theta_{0}$. By continuity, this property remains true for $\theta$ near $\theta_{0}$ and {\sl \ref{a37})} follows.

In particular, up to a reordering of the coordinates, we can assume that $d_{\rho '} F ( \rho , \theta )$ is invertible for $( \rho , \theta )$ in a neighborhood of $( \rho_{0} , \theta_{0} )$. Here $\rho '$ denotes the $2n-2$ variables $( \rho_{2} , \ldots , \rho_{2n-1} )$. Then, by the implicit function theorem, there exist a $C^{\infty}$ function $G : \R \times \S^{n-1} \to \R^{2n-2}$ such that
\begin{equation*}
\rho \in \Lambda_{+} \cap \Lambda^{+}_{\theta} \Longleftrightarrow \rho ' = G ( \rho_{1} , \theta ) .
\end{equation*}
Thus, for $\theta$ fixed, $\Lambda_{+} \cap \Lambda^{+}_{\theta}$ is locally a one dimensional manifold. Since $\Lambda_{+} \cap \Lambda^{+}_{\theta}$ is necessarily stable by the Hamiltonian flow, $\Lambda_{+} \cap \Lambda^{+}_{\theta}$ is locally a unique Hamiltonian curve and
\begin{equation*}
\rho \in \Lambda_{+} \cap \Lambda^{+}_{\theta} \Longleftrightarrow \exists t \in \R, \ \ \rho = \exp (t H_{p}) ( \rho_{0,1} , G ( \rho_{0,1} , \theta ) ) ,
\end{equation*}
locally near $\rho_{0}$ (here, $\rho_{0,1}$ can be replaced by any real close to this value). Then {\sl \ref{a38})} follows from a compactness argument on $\Lambda_{+} \cap \{ \vert x \vert = \varepsilon \}$.

Let now $z_{+}^{\ell} ( \theta )$ be the impact parameter of the trajectory $t \mapsto \exp (t H_{p}) ( \rho_{1 ,0} , G ( \rho_{1 ,0} , \theta ) )$ defined in \eqref{a49}. From \eqref{a47} and the fact that $G$ is smooth, $z_{+}^{\ell} ( \theta )$ is an $C^{\infty}$ function in $O^{+}$ if $O^{+}$ is a small enough neighborhood of $\theta_{0}$.
\end{proof}

For $m\in\{1, \dots, N_+\}$ or $m\in\{1, \dots, N_-\}$ and $\theta\in 
O^+$ or $\omega\in O^-$ we shall use the superscript $m$ to denote objects 
related to the unique trajectory $\gamma_{\pm}^{m}$ with final 
direction $\theta$ or initial direction $\omega.$
In particular, we let $S_{+}^{m}(\theta),$ $\theta\in O^+,$ denote the 
(modified) action, given by \eqref{ma}, over the $m$-th trajectory 
with final direction $\theta.$
With $S_{-}^{m}(\omega)$ for $\omega\in O^-$ defined mutatis mutandis, we 
now have the following lemma which is analogous to \cite[Lemma 5]{Al06_01} and Equation \eqref{a41}.

\begin{Lem}\sl \label{lukl}
For $m \in\{1, \dots, N_{\pm} \}$, we have $\Lambda_{S_{\pm}^{m}} = \big\{ (\alpha , \pm \partial_{\alpha} S_{\pm}^{m} ( \alpha )); \ \alpha\in O^{\pm} \big\}$.
\end{Lem}

\begin{proof}
We will only calculate $\Lambda_{S_{+}^{\ell}}$, the case of the manifold $\Lambda_{S_{-}^{k}}$ can be treated the same way. Here, we will use the notation
\begin{equation*}
\left\{ \begin{aligned}
&x_{+}(t, \theta ) = x_{+} ( t , \theta , z_{+}^{\ell} ( \theta))  \\
&\xi_{+}(t, \theta ) = \xi_{+} ( t , \theta , z_{+}^{\ell} ( \theta)) .
\end{aligned} \right.
\end{equation*}
We recall from \cite[Equation (7.11)]{AlBoRa07_01} that 
\begin{align}
\Psi_{+} \big( x_{+} ( t , \theta ) , \theta \big) =& 2 E_{0} t 1_{t>0} - \int_{t}^{+ \infty} \vert \xi_{+} (s, \theta ) \vert^2-2E_0 1_{s>0} \, d s    \nonumber   \\
=& - S_{+}^{\ell} ( \theta ) + \int_{- \infty}^{t} \vert \xi_{+} (s, \theta ) \vert^{2} d s . \label{a32}
\end{align}
From Lemma \ref{a31} and Lemma \ref{ls},
\begin{equation} \label{a34}
\Lambda_{S_{+}^{\ell}} = \big\{ \big( \theta , ( \partial_{\theta} \Psi_{+}) ( x_{+} ( t , \theta ) , \theta ) \big) ; \ \theta \in O^{+} \big\} ,
\end{equation}
for any $t \in \R$. Combining \eqref{a15} and \eqref{a32}, we obtain
\begin{align}
( \partial_{\theta} \Psi_{+}) \big( x_{+} ( t , \theta ) , \theta \big) =& \partial_{\theta} \big( \Psi_{+} ( x_{+} ( t , \theta ) , \theta ) \big) - ( \partial_{x} \Psi_{+} ) ( x_{+} ( t , \theta ) , \theta ) \partial_{\theta} ( x_{+} ( t , \theta ) )   \nonumber \\
=& - \partial_{\theta} S_{+}^{\ell} ( \theta ) + \int_{- \infty}^{t} \partial_{\theta} \big( \vert \xi_{+} (s, \theta ) \vert^{2} \big) d s - \xi_{+} (t, \theta ) \partial_{\theta} ( x_{+} ( t , \theta ) ). \label{a33}
\end{align}

Since the energy is constant on the Hamiltonian curves, we have, as in \eqref{a42},
\begin{align}
\partial_{t} \big( \xi_{+} (t, \theta ) \partial_{\theta} ( x_{+} ( t , \theta) ) \big) =& \xi_{+} (t, \theta ) \partial_{\theta} ( \xi_{+} ( t , \theta ) ) - ( \partial_{x} V) (x_{+} (t, \theta )) \partial_{\theta} ( x_{+} ( t , \theta ) )   \nonumber    \\
=& \frac{1}{2} \partial_{\theta} \big( \vert \xi_{+} (t, \theta ) \vert^{2} \big) - \partial_{\theta} \big( V (x_{+}(t, \theta )) \big)  \nonumber    \\
=& \frac{1}{2} \partial_{\theta} \big( \vert \xi_{+} (t, \theta ) \vert^{2} \big) - \partial_{\theta} \Big( E_{0} -  \frac{1}{2} \vert \xi_{+}(t, \theta ) \vert^{2} \Big)   \nonumber   \\
=& \partial_{\theta} \big( \vert \xi_{+} (t, \theta ) \vert^{2} \big) .  \label{a35}
\end{align}
Moreover, as $t \to - \infty$, we have $\xi_{+} (t, \theta) \to 0$ and
\begin{equation*}
\partial_{\theta} ( x_{+} ( t , \theta) ) = d \Pi_{x} \circ d \exp ( t H_{p} ) \big( \partial_{\theta} x_{+} (0, \theta ) , \partial_{\theta} \xi_{+} (0, \theta ) \big) \longrightarrow 0,
\end{equation*}
since $(  x_{+} (0, \theta ) ,  \xi_{+} (0, \theta ) ) \in \Lambda_{+}$ for all $\theta$ and $0$ is a unstable node of $H_{p}$ restricted to $\Lambda_{+}$. Then, \eqref{a35} yields
\begin{equation*}
\xi_{+} (t, \theta ) \partial_{\theta} ( x_{+} ( t , \theta ) ) = \int_{- \infty}^{t} \partial_{\theta} \big( \vert \xi_{+} (s, \theta ) \vert^{2} \big) d s.
\end{equation*}
Using this equality, the lemma follows from \eqref{a34} and \eqref{a33}.
\end{proof}

We now have the following

\begin{Th}\sl
Let $E=E_0 +hE_1$, with $E_1\in ]-C_{0} , C_{0} [$ for some $C_{0} >0$, and $\omega^{0} , \theta^{0} \in\mathbb{S}^{n-1}$ satisfy $\omega^{0} \ne \theta^{0}$. Then
\begin{enumerate}[i)]
\item for every $( \theta^{0} , z_{+}^{0} , \omega^{0} , z_{-}^{0} ) \in \widetilde{\Lambda_{+}^{\infty} \times \Lambda_{-}^{\infty}}$ there exist $m\in\mathbb{N}$, a symbol $a\in S_{2n+m-2}^{\frac12
-\frac{\sum_{j=1}^{n}\lambda_{j}}{2\lambda_{1}}+\frac{m}{2}}(1)$, and a 
non-degenerate phase function $\varphi\in C^{\infty}(\mathbb{R}^{2 n + m -2})$ such 
that, microlocally near $( \theta^{0} , z_{+}^{0} , \omega^{0} , z_{-}^{0} )$,
\begin{equation*}
{\mathcal{A}(E, h)} (\theta, \omega) = \int_{\mathbb{R}^{m}} 
e^{i \varphi(\theta, \omega, \tau) /h}a(\theta, \omega, \tau ;E,h) \, d \tau .
\end{equation*} \label{a36}
\medskip

\item Assume that $\Lambda_{\omega}^{-}$ intersects $\Lambda_-$ transversely and $\Lambda_{\theta}^{+}$ intersects $\Lambda_+$ transversely. For every $( \theta^{0} , z_{+}^{0} , \omega^{0} , z_{-}^{0} ) \in \widetilde{\Lambda_{+}^{\infty} \times \Lambda_{-}^{\infty}}$, there exists a symbol $a \in S_{2n-2}^{\frac12 - \frac{\sum_{j=1}^{n} \lambda_{j}}{2 \lambda_{1}}} (1)$ such that, microlocally near $( \theta^{0} , z_{+}^{0} , \omega^{0} , z_{-}^{0} )$,
\begin{equation*}
{\mathcal{A}(E, h)} (\theta , \omega ) = e^{i (S_{+} ( \theta ) + S_{-} ( \omega )) /h} a (\theta , \omega ;E,h) ,
\end{equation*}
where $S_{+} ( \theta )$ and $S_{-} ( \omega )$ are the actions defined before Lemma \ref{lukl} associated to the paths in $\Lambda_{+} \cap \Lambda^{+}_{\theta}$ and $\Lambda_{-} \cap \Lambda^{-}_{\omega}$ close to $\gamma_{+} ( t, \theta^{0} , z_{+}^{0} , E_{0} )$ and $\gamma_{-} ( t, \omega^{0} , z_{-}^{0} , E_{0} )$.
\medskip

\item Assume $O^-\cap O^+=\emptyset$ and $\< g_+ ( \rho_+ ), g_- (\rho_- ) \> \ne 0$ for all $\left(\rho_+, 
\rho_-\right)\in \Lambda_+\times\Lambda_-$ such that $\pm \lim_{t\to\pm\infty} \xi \left(t , \rho_{\pm} \right) \in \sqrt{2 E_{0}}O^{\pm}$. Let $N_{\infty}$ be the number of $(\omega, \theta)$-trajectories. For $j\in\{1, \dots, N_{\infty}\}$, $k\in\{1, \ldots , N_{-} \}$ and $\ell \in\{1, \ldots, N_{+} \}$, there exist $m_{j}, m_{k,\ell} \in \N$, non-degenerate phase functions
\begin{equation*}
\varphi_j\in 
C^{\infty}\left(\mathbb{S}^{n-1}\times\mathbb{S}^{n-1}\times\mathbb{R}^{m_{j}}   
\right) \quad \text{and} \quad \varphi_{k, \ell}\in
C^{\infty}\left(\mathbb{S}^{n-1}\times\mathbb{S}^{n-1}\times\mathbb{R}^{m_{k,\ell}}
\right),
\end{equation*}
and symbols
\begin{equation*}
a_j\in S_{2n-2+m_{j}}^{\frac{m_{j}}{2}}(1) \quad \text{and} \quad a_{k, \ell}\in S_{2n-2+m_{k, \ell}}^{\frac{1}{2} -\frac{\sum_{j=1}^{n}\lambda_{j}}{2\lambda_1} +\frac{m_{k, \ell}}{2}}(1) ,
\end{equation*}
such that
\begin{align*}
\mathcal{A}(E, h) ( \theta, \omega )
& =\sum_{j=1}^{N_{\infty}}\int_{\mathbb{R}^{m_{j}}}
e^{i \varphi_{j}(\theta, \omega, \tau) /h} a_{j}(\theta , \omega ,
\tau ;E ,h) \, d\tau\\ 
&\quad\quad +
\sum_{k=1}^{N_{-}}\sum_{\ell=1}^{N_{+}} \int_{\mathbb{R}^{m_{k, \ell}}}e^{i \varphi_{k, \ell} (\theta, \omega, \tau ) /h} a_{k, \ell}(\theta ,\omega , \tau ;E,h) \, d\tau + \mathcal{O}(h^{\infty}).
\end{align*}
\end{enumerate}
\end{Th}

\begin{proof}
{\sl i)} The first part is a direct consequence of Theorem \ref{Thsa} and \cite[Theorem 1]{Al05_02}.

{\sl ii)} The second part follows from Theorem \ref{Thsa}, Lemma \ref{lukl}, and \cite[Theorem~1]{Al05_02}.

{\sl iii)} To establish the last part of the theorem, it  suffices 
to prove that $WF_{h}^{i} ( {\mathcal{A}(E, h)} )$ $= \emptyset$. Recall the representation \eqref{ik} of the scattering amplitude. From \cite[page 166]{RoTa89_01}, we have, in the sense of oscillatory integrals,
\begin{equation*}
K_{T_{\pm 1}} = \int e^{i \psi_{\pm} (\theta , \omega ,x) /h} k_{\pm b} (x, \sqrt{2E} \omega ;h ) \overline{a_{+}} (x, \sqrt{2E} \theta ;h) \, d x ,
\end{equation*}
with $k_{\pm b} = e^{- i \Phi_{\pm} /h} \big( -\frac{h^{2}}{2} \Delta + V -\frac{1}{2} \xi^2 \big) e^{i \Phi_{\pm} /h} b_{\pm} \in A_{-1}$ and $\psi_{\pm} (\theta , \omega ,x) = \Phi_{\pm} (x, \sqrt{2 E} \omega ) - \Phi_{+} (x, \sqrt{2E} \theta )$. Since $O^{+} \cap O^{-} = \emptyset$, there exists $C>0$ such that $\vert \partial_{x} \psi_{\pm} \vert > C$ for $( \theta , \omega ) \in O^{+} \times O^{-}$. Then, integrating by parts with respect to $x$, we see that the distribution $K_{T_{\pm 1}}$ is a $C^{\infty}$ function on $O^{+} \times O^{-}$. Moreover this function and all its derivatives are bounded by ${\mathcal O} (h^{\infty})$. Therefore,
\begin{equation}\label{nowfit}
WF_{h}^{i} \big( {K_{T_{\pm 1}}}_{\vert_{O^{+} \times O^{-}}} \big) = \emptyset .
\end{equation}

From \eqref{sared}, it is clear that $(\theta , \omega ) \mapsto G ( \theta , \omega )$ is $C^{\infty}$ with respect to $(\theta , \omega )$. In some coordinate chart and for any $f_{+} ( \theta )$,  $f_{-} ( \omega )$ in $C^{\infty}_{0} ( \R^{n-1})$ supported in this chart, we have
\begin{equation*}
\big\vert \big( {\mathcal F}_{h} ( f_{+} f_{-} G ) \big) ( \xi , \eta ) \big\vert = \Big\vert \iiint e^{- i ( \xi \theta + \eta \omega )/h} e^{- i \Phi_{+} ( x, \sqrt{2E_{0}} \theta )/h} f_{+} \overline{g_{+}} \CR \big( e^{i \Phi_{-} ( y, \sqrt{2E_{0}}\omega )/h} f_{-} g_{-} \big) \, d x \, d \theta \, d \omega \Big\vert ,
\end{equation*}
for $\xi , \eta \in \R^{n-1}$. For $\vert \xi \vert$ large enough, we have
\begin{equation*}
\vert \partial_{\theta} ( \xi \theta +\Phi_{+} ( x, \sqrt{2E_{0}} \theta ) ) \vert \gtrsim \< \xi \> ,
\end{equation*}
on the support of $\overline{g_{+}}$. The same way, for $\vert \eta \vert$ large enough, we have
\begin{equation*}
\vert \partial_{\omega} ( \eta \omega - \Phi_{-} ( y, \sqrt{2E_{0}}\omega ) ) \vert \gtrsim \< \eta \>,
\end{equation*}
on the support of $g_{-}$. Then, performing integrations by parts with respect to $\theta$ or $\omega$, we obtain $\vert ( {\mathcal F}_{h} (f_{+} f_{-} G) ) ( \xi , \eta ) \vert = {\mathcal O} ( h^{\infty} \< \xi , \eta \>^{- \infty} )$ for $\< \xi , \eta \>$ large enough, and then
\begin{equation}\label{nowfig}
WF_{h}^{i}\left(G\right)=\emptyset.
\end{equation}

To treat, now, the terms in \eqref{t2} containing the operators whose norms are estimated 
in Lemma \ref{l2.1RT}, we use the following lemma, the proof 
of which we present later.

\begin{Lem}\sl \label{nowfi}
Let $T \in \mathcal{B} (L^{2}_{-\gamma}(\mathbb{R}^{n}), L^{2}_{\gamma}(\mathbb{R}^{n}) )$ satisfy $\Vert T \Vert_{\mathcal{B} (L^{2}_{-\gamma} , L^{2}_{\gamma} )} = \mathcal{O} (h^{\infty})$ for all $\gamma\gg 1$ and let $E >0$. Then
\begin{equation*}
WF_{h}^{i} \big( K_{F_{0}(E, h)T F_{0}^{*}(E, h)} \big) = \emptyset.
\end{equation*}
\end{Lem}

From \eqref{ik}, \eqref{nowfit}, \eqref{t2}, \eqref{nowfig}, Lemma \ref{l2.1RT}, and Lemma 
\ref{nowfi} we now conclude that
\begin{equation} \label{a40}
WF_{h}^{i} ( \mathcal{A}(E, h) )= \emptyset.
\end{equation}

For $j \in\{1, \dots, N_{\infty}\}$ we now let $\CS\CR_{j} (E_0)$ denote the scattering relation near the $j$-th $(\omega, \theta )$-trajectory, defined in \eqref{a52} and indicated in Figure \ref{fig:sr}.
From Theorem \ref{nontrappedcase},
\begin{equation*}
\CS (E,h)\in\mathcal{I}_{h}^{0}\left(\mathbb{S}^{n-1}\times\mathbb{S}^{n-1},
\CS\CR_{j} (E_{0} ) '\right),
\end{equation*}
microlocally near the limit points of the $j$-th $(\omega, \theta )$-trajectory. From \eqref{a40}, it is enough to know the scattering amplitude microlocally in a compact set. Then, the conclusion of the theorem follows from these observations, \eqref{a53}, Theorem \ref{Thsa} and \cite[Theorem 1]{Al05_02}.
\end{proof}

\begin{proof}[Proof of Lemma \ref{nowfi}]
In some coordinate chart and for any $f_{+} ( \theta )$,  $f_{-} ( \omega )$ in $C^{\infty}_{0} ( \R^{n-1})$ supported in this chart, we have
\begin{align*}
K ( \xi , \eta ) =& \big( {\mathcal F}_{h} K_{f_{+} F_{0} (E ,h) T F_{0}^{*}(E , h) f_{-}} \big) ( \xi , \eta )   \\
=& c_{2} \iiint e^{ - i ( \theta \xi + \sqrt{2 E} x \theta )/h} f_{+} ( \theta ) T \big( e^{ i ( \sqrt{2 E} y \omega - \omega \eta )/h} f_{-} ( \omega )\big) \, d x \, d \theta \, d \omega .
\end{align*}
with $c_{2} = (2\pi h)^{- n} (2E)^{\frac{n-2}{2}}$. In particular, for $\alpha , \beta \in \N^{n-1}$,
\begin{align*}
\xi^{\alpha} \eta^{\beta} K ( \xi , \eta ) = c_{2} \iiint e^{ - i \theta \xi /h} (- i h \partial_{\theta} )^{\alpha} &\big( e^{-i \sqrt{2 E} x \theta /h} f_{+} ( \theta ) \big)    \\
&T \big( e^{ -i  \omega \eta /h} (- i h \partial_{\omega} )^{\beta} \big( e^{ i \sqrt{2 E} y \omega /h} f_{-} ( \omega ) \big) \big) \, d x \, d \theta \, d \omega ,
\end{align*}
We remark that
\begin{gather*}
e^{ - i \theta \xi /h} (- i h \partial_{\theta} )^{\alpha} \big( e^{-i \sqrt{2 E} x \theta /h} f _{+}( \theta ) \big) \in L^{2}_{-n/2 -1 - \vert \alpha \vert} ( \R^{n}_{x}) ,  \\
e^{ -i  \omega \eta /h} (- i h \partial_{\omega} )^{\beta} \big( e^{ i \sqrt{2 E} y \omega /h} f_{-} ( \omega ) \big) \in L^{2}_{-n/2 -1 - \vert \beta \vert} ( \R^{n}_{y}) ,
\end{gather*}
uniformly with respect to $h , \xi , \eta , \theta , \omega$. Combining $\Vert T \Vert_{\mathcal{B} (L^{2}_{-n/2 -1 - \vert \beta \vert} , L^{2}_{n/2 + 1 + \vert \alpha \vert} )} = \mathcal{O} (h^{\infty})$ with these estimates and the compacity of $\S^{n-1}$, we get
\begin{equation*}
\xi^{\alpha} \eta^{\beta} K ( \xi , \eta ) = \CO (h^{\infty}),
\end{equation*}
uniformly in $\xi , \eta$ and the lemma follows.
\end{proof}

\begin{remark}\sl
It is clear that all estimates in the above proof can be made uniform in the energy if that is allowed to vary in a bounded set.
\end{remark}

\appendix
\section{Elements of semi-classical analysis}\label{scanal}

\Subsection{Semi-classical distributions}

Here we recall some of the elements of semi-classical analysis which we use along the paper.
A family $(u_{h})_{h\in]0,h_{0}]}$ of distributions in $\CD'(\R^n)$  is called a semi-classical distribution
when 
\begin{equation*}
\forall \chi \in C_{0}^{\infty} ({\R}^{n}), \quad  \exists N \in \N, \quad \mathcal{F}_{h}(\chi u) (\xi)
\lesssim h^{-N}\langle \xi \rangle^{N},
\end{equation*}
where $\mathcal{F}_{h}$ is the $h$-Fourier transform
\begin{equation*}
\mathcal{F}_{h}(\chi u)(\xi)= \int_{\R^{n}} e^{- i x \cdot \xi /h} \chi u (x) \, d x .
\end{equation*}
The space of semi-classical distributions is denoted $\mathcal{D}_{h}'(\mathbb{R}^{n})$. We define the semi-classical wavefront set of $u=(u_{h})\in\mathcal{D}_{h}'(\mathbb{R}^{n})$ as follows.

\begin{Def}\label{defM}\sl 
Let $u\in\mathcal{D}'_{h}({\R}^{n})$ and let $\left(x_{0}, \xi_{0}\right)\in T^{*}\mathbb{R}^{n} \sqcup T^{*} \mathbb{S}^{n-1}$. We shall say that $\left(x_{0}, \xi_{0}\right)$ does not belong to the
semi-classical wavefront set of $u$ if:
\begin{itemize}
\item If $\left(x_{0}, \xi_{0}\right)\in T^{*}\mathbb{R}^{n}$: 
there exist $\chi\in
C_{0}^{\infty}\left(\mathbb{R}^{n}\right)$ with $\chi\left(x_{0}\right)\ne
0$ and an open neighborhood
$U$ of $\xi_{0}$, such that $\forall N\in\mathbb{N},$ $\forall\xi\in U$,
\begin{equation*}
\vert \mathcal{F}_{h}\left(\chi u\right)\left(\xi\right) \vert \leq C_{N,
U}h^{N}.
\end{equation*}
We shall denote the complement of the set of all such points by 
$WF_{h}^{f}(u).$

\medskip

\item If $\left(x_{0}, \xi_{0}\right)\in T^{*}\mathbb{S}^{n-1}$: there 
exist $\chi\in
C_{0}^{\infty}\left(\mathbb{R}^{n}\right)$ with $\chi\left(x_{0}\right)\ne
0$ and a
conic neighborhood $U$ of $\xi_{0}$, such that $\forall N\in\mathbb{N},$
$\forall\xi\in
U\cap\left\{\vert \xi \vert \geq \frac{1}{K}\right\}$ for some $K>0,$
\[|\mathcal{F}_{h}\left(\chi u\right)\left(\xi\right)|\leq
C_{N, U, K}h^{N}\left\langle\xi\right\rangle^{-N}.\]
We shall denote the complement of the set of all such points by 
$WF_{h}^{i}(u).$
\end{itemize}
We shall further use $WF_{h}\left(u\right)=WF_{h}^{f}(u)\sqcup 
WF_{h}^{i}(u)$ to denote the
semi-classical
wavefront set of $u.$
\end{Def}

A family $(u_{h})_{h\in]0,h_{0}]}$ of temperate distributions in $\CS '(\R^n)$ is called a semi-classical temperate distribution when, for some $N \in \R$,
\begin{equation*}
\< x, h D \>^{-N} u = \CO (h^{-N} ) ,
\end{equation*}
in $L^{2} (\R^{n})$. The space of semi-classical temperate distributions is denoted $\CS_{h} '( \R^{n})$.

\Subsection{Pseudodifferential operators}

We now define briefly the semi-classical pseudodifferential operators (see the book of Dimassi and Sj\"{o}strand \cite{DiSj99_01}).
A positive function $m : \R^{p} \to ]0, + \infty [$ is called an {\it order function} if there exists $C>0$ such that
\begin{equation*}
m (X) \leq C \< X-Y \>^{C} m (Y),
\end{equation*}
for all $X,Y \in \R^{p}$. We denote by $S_{p}^{q} (m)$ the set of (families of) functions $a (X;h) \in C^\infty(\R^{p})$ such that, for all $\alpha \in \N^{p}$,
\begin{equation*}
\partial^\alpha_{X} a (X ;h)=\CO(h^{-q} m(X)).
\end{equation*}
If $a (x,\xi ;h)$ is a symbol of class $S_{2n}^{q} (m)$, we define the $h$-pseudodifferential operator, in Weyl quantization, $\Op (a)$ with symbol $a$ by
\begin{equation}
\forall u \in \CS(\R^{n}),\quad  \left( \Op (a) u \right) (x) = \frac{1}{(2 \pi h)^{n}} \iint e^{i(x-y)\cdot \xi /h}a \Big( \frac{x+y} 2, \xi ; h \Big) u(y) \, d y \, d \xi,
\label{oph}
\end{equation}
extending the definition to $\mathcal{S}'\left(\mathbb{R}^{n}\right)$
by duality.
We also denote by $\Psi^{q} (m)$ the space of operators $\Op (S^{q}_{2n} (m))$.

We extend these notions to compact manifolds through the following definition of semi-classical pseudodifferential operators on compact manifolds.
Let $M$ be a smooth compact manifold and $\kappa_{j}: M_j\to X_j,$ $j=1, \ldots, N,$ be a set of
local charts.
A linear continuous operator $A: C^{\infty}(M) \to \mathcal{D}_{h}'(M)$ belongs to $\Psi^{q} (1, M)$ if for all $j \in \{ 1, \dots, N\}$ and $u \in C^{\infty}_{0}(M_{j})$ we have $Au \circ \kappa_{j}^{-1}=A_j \big( u \circ \kappa_{j}^{-1} \big)$ with $A_j \in \Psi^{q} (1)$, and $\chi_1 A \chi_2 : \mathcal{D}_{h}'(M) \to h^{\infty} C^{\infty} (M)$ if $\supp \chi_1 \cap \supp\chi_2 = \emptyset$ (see \cite[Section E.2]{EvZw07_01} for more details).

\Subsection{Microlocal Properties}

We can now define that we mean by ``microlocally''. We will only work on $\R^{n}$. Using the previous paragraph, this definition can be extended to the case of compact manifolds.

Let $u , v \in \CS_{h}'( \R^{n} )$. We say that $u=v$ {\it microlocally} near a set $U\subset T^{*} \mathbb{R}^{n}$, if there exist $a \in S^{0} (1)$, $a =1$ in a neighborhood of $U$, such that
\begin{equation*}
\Op (a) ( u - v) = \CO (h^{\infty} ),
\end{equation*}
in $L^{2} ( \R^{n} )$. We also say that $u \in \CS_{h}'( \R^{n} )$ satisfies a property $\mathcal{P}$  {\it microlocally} near a set $U\subset T^{*} \R^{n}$ if there exist $v \in \CS_{h}' ( \R^{n} )$ such that $u=v$
microlocally near $U$ and $v$ satisfies property $\mathcal{P}$.

\begin{Def}\sl
Let $A , B : L^{2}( \R^{n}) \rightarrow L^{2} ( \R^{m})$ be linear operators bounded by $\CO (h^{- N})$, $N>0$ and $( \rho , \widetilde{\rho} ) \in T^{*} \R^{m} \times T^{*} \R^{n}$. We say that
\begin{equation*}
A =B \text{ microlocally near } ( \rho , \widetilde{\rho} ),
\end{equation*}
if there exists $\alpha \in C^{\infty}_{0} (T^{*} \R^{m})$ (resp. $\beta \in C^{\infty}_{0} (T^{*} \R^{n})$) equal to $1$ near $\rho$ (resp. $\widetilde{\rho}$) such that
\begin{equation*}
\Op ( \alpha ) ( B -A) \Op ( \beta ) = \CO (h^{\infty}),
\end{equation*}
in ${\mathcal B} (L^{2}( \R^{n}) , L^{2} ( \R^{m}))$.
\end{Def}

\Subsection{Semi-classical Fourier integral operators}

We now define global semi-classical Fourier integral operators. For the general theory of the FIOs in the classical setting, we refer to H\"{o}rmander \cite[Section 25.2]{Ho94_01}. The theory of the semi-classical FIOs can be found in the books of Ivrii \cite[Section 1.2]{Iv98_01}, Robert \cite{Ro87_01}, in the PhD thesis of Dozias \cite{Do94_01} or in the article of the first author \cite{Al05_02}. We will develop with theory in $\R^{n}$. Using local charts, the following definitions and theorem can easily be extended to the case of compact manifolds.

Let $\varphi (x, y, \theta ) \in C^{\infty} ( \Omega )$ where $\Omega$ is an open set of $\R^{m+n+d}$. We say that $\varphi$ is a {\it non-degenerate phase function} if $d \varphi \neq 0$ everywhere in $\Omega$ and, for all $(x, y , \theta ) \in C_{\varphi}$ with
\begin{equation*}
C_{\varphi} = \{ (x, y , \theta) \in \Omega ; \ \partial_{\theta} \varphi = 0 \} ,
\end{equation*}
the $d$ differentials $d \partial_{\theta_{1}} \varphi , \ldots , d \partial_{\theta_{d}} \varphi$ are linearly independant.

If $\varphi$ is a non-degenerate phase function, $C_{\varphi}$ is a $(m+n)$-dimensional manifold and
\begin{equation*}
\begin{aligned}
j_{\varphi} : \\
{}
\end{aligned}
\left\{ \begin{aligned}
&C_{\varphi} \longrightarrow && T^{*} \R^{m} \times T^{*} \R^{n} \\
&(x, y , \theta) && (x, \partial_{x} \varphi , y , \partial_{y} \varphi)
\end{aligned} \right. 
\end{equation*}
is locally a diffeomorphism whose image is a Lagrangian manifold for the symplectic form $d \xi \wedge d x + d \eta \wedge d y$ ($(x, \xi)$ and $(y , \eta )$ are the standard coordinates on $T^{*} \R^{m}$ and $T^{*} \R^{n}$). We note $\Lambda_{\varphi} = j_{\varphi} (C_{\varphi})$.

\begin{Def}\sl
A submanifold $\Lambda \subset T^{*} \R^{m} \times T^{*} \R^{n}$ is a canonical relation if $\Lambda$ is a Lagrangian manifold for the symplectic form $d \xi \wedge d x - d \eta \wedge d y$.

A canonical relation $\Lambda$ is given by a canonical transformation if there exists a symplectic diffeomorphism $\kappa : T^{*} \R^{n} \to T^{*} \R^{m}$ such that $\Lambda = \graph ( \kappa )$.
\end{Def}

As usual, if $\Lambda \subset T^{*} \R^{m} \times T^{*} \R^{n}$, we note
\begin{equation*}
\Lambda ' = \{ (x,\xi , y , - \eta ) ; \ (x,\xi , y , \eta ) \in \Lambda \} ,
\end{equation*}
the subset of $T^{*} \R^{m} \times T^{*} \R^{n}$. In particular, for a non-degenerate phase function $\varphi$, the manifold $\Lambda_{\varphi} '$ is a canonical relation (if $\varphi$ is restricted to a small set).

\begin{Def}\sl \label{a44}
Let $r \in \R$, $\Lambda$ be a canonical relation from $T^{*} \R^{n}$ to $T^{*} \R^{m}$ and $A : L^{2}( \R^{n}) \rightarrow L^{2} ( \R^{m})$ be a linear operator bounded by $\CO (h^{- N})$, $N>0$. Then, $A$ is called a $h$-Fourier integral operator ($h$-FIO's) of order $r$ associated to $\Lambda$ and we note
\begin{equation*}
A \in \mathcal{I}_{h}^{r} ( \R^{m} \times \R^{n}, \Lambda ') ,
\end{equation*}
if, for all $(\rho , \widetilde{\rho} ) \in T^{*} \R^{m} \times T^{*} \R^{n}$, $A$ is equal to
\begin{equation} \label{a45}
h^{-r - \frac{n+m}{4} - \frac{d}{2}} \int_{\theta \in \R^{d}} e^{i \varphi (x, y , \theta ) /h} a (x, y , \theta ;h) \, d \theta.
\end{equation}
microlocally near $( \rho , \widetilde{\rho} )$. Here, the symbol $a \in S^0(1)$ has compact support in the variables $x, y , \theta$ (uniformly with respect to $h$). The function $\varphi$ is a non-degenerate phase function defined near the support of $a$ with $\Lambda_{\varphi} {}' \subset \Lambda$.

A $h$-FIO $A$ will be called a $h$-Fourier integral operator with compactly supported symbol if, modulo an operator $\CO (h^{\infty})$ in ${\mathcal B} (L^{2}( \R^{n}) , L^{2} ( \R^{m}))$, $A$ is a finite sum of operators of the form \eqref{a45}.
\end{Def}

Lastly, we give the composition law for $h$-Fourier integral operators (see {\it e.g.} \cite{Do94_01} for the proof). The following theorem is a semi-classical version of Theorem 25.2.3 of H\"{o}rmander \cite{Ho94_01}. Since all the $h$-FIOs which appear in this paper (except the one in Lemma \ref{lem1}) have compactly supported symbol, we give the composition law only in that case.

Let $A_{1}\in \mathcal{I}_{h} (\R^m \times \R^n,\Lambda_{1} {}')$ and $A_{2}\in \mathcal{I}_{h}^{r_{2}} (\R^n \times \R^p , \Lambda_{2} {}')$ be two $h$-FIO's with compactly supported symbols, associated with $\Lambda_{1}\subset T^*\R^n\times T^*\R^m$ and $\Lambda_{2}
 \subset T^*\R^m\times T^*\R^p$ respectively.  We set
\begin{align*}
& X=T^*\R^n\times T^*\R^m \times T^*\R^m\times T^*\R^p   \\
& Y=\Lambda_{1}\times \Lambda_{2}\subset  X      \\
& Z=T^*\R^n\times \diag(T^*\R^m \times T^*\R^m) \times T^*\R^p\subset X.
\end{align*}

\begin{Def}\sl
We say that $Y$ and $Z$ intersect cleanly if $Y \cap Z$ is a manifold and $T_{\rho}(Y\cap Z)=T_{\rho}Y\cap T_{\rho}Z$ at each $\rho\in Y\cap Z$. The excess of the intersection is 
\begin{equation*}
e= \dim X+\dim Y\cap Z-\dim Y -\dim Z.
\end{equation*}
\end{Def}

Let
\begin{equation} \label{a46}
\pi : Y \cap Z \to T^{*} \R^{m} \times T^{*} \R^{p} ,
\end{equation}
be the natural projection. The image of $\pi$ is
\begin{equation*}
\Lambda_{2} \circ \Lambda_{1} = \{ (\rho_{3} , \rho_{1} ) \in T^{*} \R^{m} \times T^{*} \R^{p} ; \ \exists \rho_{2} \in T^{*} \R^{n} , \ (\rho_{3} , \rho_{2} ) \in \Lambda_{2} \text{ and } (\rho_{2} , \rho_{1} ) \in \Lambda_{1} \}.
\end{equation*}

\begin{Def}\sl
We say that $Y$ and $Z$ intersect connectedly if, for all $\gamma \in T^{*} \R^{m} \times T^{*} \R^{p}$, the set $\pi^{-1} ( \gamma )$ is connected.
\end{Def}

When $Y$ and $Z$ intersect cleanly and connectedly, the set $\Lambda_{2} \circ \Lambda_{1}$ is a Lagrangian submanifold of $T^{*} \R^{m} \times T^{*} \R^{p}$. In general, the intersection $Y \cap Z$ is also assumed to be proper. This means that $\pi$, defined in \eqref{a46}, is proper. But since $A_{1}$ and $A_{2}$ have compactly supported symbol, we don't have to make such hypothesis.

\begin{Th}\sl \label{composeFIO}
Let $A_{1}\in \mathcal{I}_{h} (\R^m \times \R^n,\Lambda_{1} {}')$ and $A_{2}\in \mathcal{I}_{h}^{r_{2}} (\R^n \times \R^p , \Lambda_{2} {}')$ be two $h$-FIOs with compactly supported symbols. If $Y$ and $Z$ intersect connectedly and cleanly with excess $e$, then
\begin{equation*}
A_{2}\circ A_{1}\in\CI^{r_{1}+r_{2}+e/2}_{h} (\R^m\times\R^p, \Lambda_{2} \circ \Lambda_{1} {}') ,
\end{equation*}
is a $h$-FIO with compactly supported symbol.
\end{Th}

As stated in \cite[Page 18]{Ho94_01}, the hypothesis ``$Y$ and $Z$ intersect connectedly'' is made to avoid self-intersections of $\Lambda_{2} \circ \Lambda_{1}$. In particular, this assumption can be replaced by ``$\Lambda_{2} \circ \Lambda_{1}$ is a manifold''. Note that all the compositions in this paper satisfies this last statement.

\providecommand{\bysame}{\leavevmode\hbox to3em{\hrulefill}\thinspace}
\providecommand{\MR}{\relax\ifhmode\unskip\space\fi MR }
\providecommand{\MRhref}[2]{%
  \href{http://www.ams.org/mathscinet-getitem?mr=#1}{#2}
}
\providecommand{\href}[2]{#2}


\begin{thebibliography}{10}

\bibitem{AbMa78_01}
R.~Abraham and J.~Marsden, \emph{Foundations of mechanics}, Benjamin/Cummings
  Publishing Co. Inc. Advanced Book Program, Reading, M., 1978, Second edition,
  revised and enlarged, With the assistance of T.~Ra\c tiu and R. Cushman.

\bibitem{Al05_02}
I.~Alexandrova, \emph{Semi-classical wavefront set and fourier integral
  operators}, preprint math/0407460 on arxiv.org. (2005).

\bibitem{Al05_01}
I.~Alexandrova, \emph{Structure of the semi-classical amplitude for general scattering
  relations}, Comm. Partial Differential Equations \textbf{30} (2005),
  no.~10-12, 1505--1535.

\bibitem{Al06_02}
I.~Alexandrova, \emph{Semi-classical behavior of the spectral function}, Proc. Amer.
  Math. Soc. \textbf{134} (2006), no.~8, 2295--2302 (electronic).

\bibitem{Al06_01}
I.~Alexandrova, \emph{Structure of the short range amplitude for general scattering
  relations}, Asymptotic Analysis (2006), no.~50, 13--30.

\bibitem{AlBoRa07_01}
I.~Alexandrova, J.-F. Bony, and T.~Ramond, \emph{Semiclassical scattering
  amplitude at the maximum point of the potential}, preprint 0704.1632 on
  arxiv.org. (2007).

\bibitem{BoFuRaZe07_01}
J.-F. Bony, S.~Fujii\'e, T.~Ramond, and M.~Zerzeri, \emph{Microlocal kernel of
  pseudodifferential operators at a hyperbolic fixed point}, J. Funct. Anal.
  \textbf{252} (2007), no.~1, 68--125.

\bibitem{DiSj99_01}
M.~Dimassi and J.~Sj{\"o}strand, \emph{Spectral asymptotics in the
  semi-classical limit}, London Mathematical Society Lecture Note Series, vol.
  268, Cambridge University Press, Cambridge, 1999.

\bibitem{Do94_01}
S.~Dozias, \emph{Op\'{e}rateurs h-pseudodiff\'{e}rentiels \`{a} flot
  p\'{e}riodique}, Ph. D. Thesis., Universit\'{e} Paris Nord, 1994.

\bibitem{EvZw07_01}
L.~Evans and M.~Zworski, \emph{Lectures on semiclassical analysis}, version
  0.3, preprint avaible at {\tt http://math.berkeley.edu/$\sim$zworski/}, 2007.

\bibitem{GeMa89_01}
C.~G{\'e}rard and A.~Martinez, \emph{Semiclassical asymptotics for the spectral
  function of long-range {S}chr\"odinger operators}, J. Funct. Anal.
  \textbf{84} (1989), no.~1, 226--254.

\bibitem{Gu76_01}
V.~Guillemin, \emph{Sojourn times and asymptotic properties of the scattering
  matrix}, Proceedings of the Oji Seminar on Algebraic Analysis and the RIMS
  Symposium on Algebraic Analysis (Kyoto Univ., Kyoto, 1976), vol.~12, 1976/77
  supplement, pp.~69--88.

\bibitem{HaWu07_01}
A.~Hassell and J.~Wunsch, \emph{The semiclassical resolvent and the propagator
  for nontrapping scattering metrics}, preprint math/0606606 on arxiv.org.
  (2006).

\bibitem{HeSj85_01}
B.~Helffer and J.~Sj{\"o}strand, \emph{Multiple wells in the semiclassical
  limit. {III}. {I}nteraction through nonresonant wells}, Math. Nachr.
  \textbf{124} (1985), 263--313.

\bibitem{Ho94_01}
L.~H{\"o}rmander, \emph{The analysis of linear partial differential operators.
  {IV}}, Grundlehren der Mathematischen Wissenschaften, vol. 275,
  Springer-Verlag, Berlin, 1994, Fourier integral operators, Corrected reprint
  of the 1985 original.

\bibitem{IsKi85_01}
H.~Isozaki and H.~Kitada, \emph{Modified wave operators with time-independent
  modifiers}, J. Fac. Sci. Univ. Tokyo Sect. IA Math. \textbf{32} (1985),
  no.~1, 77--104.

\bibitem{IsKi86_01}
H.~Isozaki and H.~Kitada, \emph{Scattering matrices for two-body {S}chr\"odinger operators},
  Sci. Papers College Arts Sci. Univ. Tokyo \textbf{35} (1986), no.~2, 81--107.

\bibitem{Iv98_01}
V.~Ivrii, \emph{Microlocal analysis and precise spectral asymptotics},
  Springer Monographs in Mathematics, Springer-Verlag, Berlin, 1998.

\bibitem{Mi04_01}
L.~Michel, \emph{Semi-classical behavior of the scattering amplitude for
  trapping perturbations at fixed energy}, Canad. J. Math. \textbf{56} (2004),
  no.~4, 794--824.

\bibitem{Po85_01}
G.~Popov, \emph{Spectral asymptotics for elliptic second order differential
  operators}, J. Math. Kyoto Univ. \textbf{25} (1985), no.~4, 659--681.

\bibitem{PoSh83_01}
G.~Popov and M.~Shubin, \emph{Asymptotic expansion of the spectral function for
  second-order elliptic operators in {${\bf R}\sp{n}$}}, Funktsional. Anal. i
  Prilozhen. \textbf{17} (1983), no.~3, 37--45.

\bibitem{Pr82_01}
Y.~Protas, \emph{Quasiclassical asymptotic behavior of the scattering amplitude
  of a plane wave on the inhomogeneities of a medium}, Mat. Sb. (N.S.)
  \textbf{117(159)} (1982), no.~4, 494--515, 560.

\bibitem{Ra96_01}
T.~Ramond, \emph{Semiclassical study of quantum scattering on the line}, Comm.
  Math. Phys. \textbf{177} (1996), no.~1, 221--254.

\bibitem{ReSi79_01}
M.~Reed and B.~Simon, \emph{Methods of modern mathematical physics. {III}},
  Academic Press, New York, 1979, Scattering theory.

\bibitem{Ro87_01}
D.~Robert, \emph{Autour de l'approximation semi-classique}, Progress in
  Mathematics, vol.~68, Birkh\"auser Boston Inc., Boston, MA, 1987.

\bibitem{RoTa88_01}
D.~Robert and H.~Tamura, \emph{Semi-classical asymptotics for local spectral
  densities and time delay problems in scattering processes}, J. Funct. Anal.
  \textbf{80} (1988), no.~1, 124--147.

\bibitem{RoTa89_01}
D.~Robert and H.~Tamura, \emph{Asymptotic behavior of scattering amplitudes in semi-classical
  and low energy limits}, Ann. Inst. Fourier (Grenoble) \textbf{39} (1989),
  no.~1, 155--192.

\bibitem{Va84_01}
B.~R. Va{\u\i}nberg, \emph{Complete asymptotic expansion of the spectral
  function of second-order elliptic operators in {${\bf R}\sp{n}$}}, Mat. Sb.
  (N.S.) \textbf{123(165)} (1984), no.~2, 195--211.

\bibitem{Va89_01}
B.~R. Va{\u\i}nberg, \emph{Asymptotic methods in equations of mathematical physics}, Gordon
  \& Breach Science Publishers, New York, 1989, Translated from the Russian by
  E. Primrose.

\bibitem{Va07_01}
A.~Vasy, \emph{The wave equation on asymptotically de sitter-like spaces},
  preprint 0706.3669 on arxiv.org. (2007).

\bibitem{Ya87_01}
K.~Yajima, \emph{The quasiclassical limit of scattering amplitude. {$L\sp
  2$}-approach for short range potentials}, Japan. J. Math. (N.S.) \textbf{13}
  (1987), no.~1, 77--126.

\end{thebibliography}
\end{document}